\documentclass[11pt]{amsart}
\usepackage[latin1]{inputenc}
\usepackage{amssymb}
\usepackage{epsfig}
\usepackage[table]{xcolor}
\newtheorem{thm}{Theorem}[section]
\newtheorem{cor}[thm]{Corollary}
\newtheorem{conj}[thm]{Conjecture}
\newtheorem{lem}[thm]{Lemma}
\newtheorem{exa}[thm]{Example}

\newtheorem{pro}[thm]{Proposition}
\theoremstyle{definition}
\newtheorem{defi}[thm]{Definition}
\newtheorem{rem}[thm]{Remark}

\newtheorem{que}[thm]{Question}

\newtheoremstyle{break}
  {\topsep}{\topsep}%
  {\itshape}{}%
  {\bfseries}{}%
  {\newline}{}%
\theoremstyle{break}
\newtheorem{theorem}[thm]{Theorem}

\usepackage{arydshln}
\usepackage[normalem]{ulem}

\begin{document}
\title[Hopf triangulations and equilibrium triangulations]
{Hopf triangulations of spheres and equilibrium triangulations
of projective spaces}
\author[Wolfgang K\"uhnel \& Jonathan Spreer]{Wolfgang K\"uhnel \& Jonathan Spreer}
\thanks{}
\subjclass[2020]{Primary: 57Q15; Secondary: 52B15, 52B70, 57Q91}
\keywords{Real projective space, Complex projective space, triangulations of manifolds, equilibrium triangulations, Hopf fibrations}

\begin{abstract}

{\small Following work by the first author and Banchoff, we investigate 
triangulations of real and complex projective spaces of real and complex
dimension $k$ that are adapted to the decomposition into ``zones of 
influence'' around the points $[1,0,\ldots,0],$ $\ldots,$ $[0,\ldots,0,1]$
in homogeneous coordinates. 
The boundary of such a ``zone of influence''
must admit a simplicial version of the Hopf decomposition of a sphere into
``solid tori'' of various dimensions. We present such {\em Hopf triangulations}
of $S^{2k-1}$ for $k \leq 4$, and give candidate triangulations for 
arbitrary $k$.

In the complex case, a crucial role of this construction is the central 
$k$-torus as the intersection of all ``zones of influence''.
Candidate triangulations of the $k$-torus with $2^{k+1}-1$, $k\geq 1$,
vertices -- possibly the minimum numbers -- are well known. They admit an
involution acting like complex conjugation and an automorphism of order 
$k+1$ realising the cyclic shift of coordinate directions in $\mathbb{C}P^k$.
For $k=2$, this can be extended to what we call a {\em perfect equilibrium 
triangulation} of $\mathbb{C}P^2$, previously described in the literature. 
We prove that this is no longer possible for $k=3$, and no perfect equilibrium
triangulation of $\mathbb{C}P^3$ exists.
In the real case, the central torus is replaced by its fixed-point set 
under complex conjugation: the vertices of a $k$-dimensional cube. We revisit
known equilibrium triangulations of $\mathbb{R}P^k$ for $k\leq 2$, and describe
new equilibrium triangulations of $\mathbb{R}P^3$ and $\mathbb{R}P^4$. 

Finally, we discuss the most symmetric and vertex-minimal triangulation
of $\mathbb{R}P^4$ and present a tight polyhedral embedding
of $\mathbb{R}P^3$ into 6-space. No such embedding was known before.}
\end{abstract}
\maketitle

\section{Introduction}  
\label{sec:intro}

In geometric topology, manifolds are often studied by first decomposing them into simplices. Such {\em triangulations} allow for a finite description of the underlying space, and are a suitable starting point for manual, as well as computer-aided calculations. Triangulations of manifolds are most useful, when their combinatorial properties simplify a topological analysis. As a result, a large body of literature is concerned with describing triangulations of manifolds that are, for one reason or another, particularly ``nice'' to work with. 

Here, ``nice'' can mean that triangulations are highly symmetric, see \cite{Bu1,Casella,CD,KoLu,Sp1,Sp2}. It can also mean that they are very economical in how many simplices they use, see \cite{BasakDatta,BasakSpreer,FGGTV,IshikawaNemoto,lensspaces,torusbundles,Ku2,KL}, but also -- again -- \cite{Casella}. However, here we focus on a third type of ``nice'' triangulations. Namely, triangulations with combinatorial features revealing topological properties of the underlying manifolds. This, as well, has been done in various contexts and dimensions \cite{BK,Bu1,JR,KleeNovik,KuLu}. Arguably most relevant for us are triangulations that respect decompositions of a manifold into handlebodies. For an example, consider the framework of {\em layered triangulations} \cite{JR} of arbitrary 3-manifolds, and in particular lens spaces, respecting a given Heegaard splitting. 

It is worthwhile mentioning that the above list of articles covers work on many different types of triangulations and simplicial cell-complexes. 
This article, however, is written from the viewpoint of combinatorial topology, where piecewise linear (PL) manifolds are triangulated as abstract simplicial complexes supporting the PL structure of the underlying space. 

\medskip

We focus on describing triangulations of complex and real projective spaces that respect the structure of their homogeneous coordinates: $\mathbb{C}P^k$ and $\mathbb{R}P^k$ can be decomposed into $2^{k+1}-1$ subsets, one for each non-empty subset of homogeneous coordinates being dominant in absolute value. We call a triangulation containing these subsets as subcomplexes an {\em equilibrium triangulation}. Such equilibrium triangulations of low-dimensional projective spaces have been described in the literature \cite{BK}. This concept has been generalised in \cite{BasakSarkar,DattaSarkar}. Here, we survey the results on equilibrium triangulations of projective spaces, and describe new such triangulations.
	
\medskip

In more detail, we introduce the notion of {\em Hopf triangulations} of spheres (Section~\ref{sec:hopf}) as triangulations respecting the Hopf fibration of odd-dimensional spheres in their combinatorial structure. We also formalise the notion of {\em equilibrium triangulations} of projective spaces (Sections~\ref{sec:equilcomplex} and \ref{sec:equilreal}). In this setup, Hopf triangulations of spheres occur in equilibrium triangulations of complex projective spaces as boundaries of ``zones of influence'' with a single coordinate direction being dominant in absolute value. 

We also introduce what we call a {\em perfect equilibrium triangulation} of $\mathbb{C}P^k$ (Section~\ref{sec:equilcomplex}) containing the smallest known triangulation of the central torus, and respecting complex conjugation and coordinate shifts as combinatorial automorphisms. Moreover, we define what we call {\em nice equilibrium triangulations} of $\mathbb{R}P^k$ (Section~\ref{sec:equilreal}) as triangulations using -- in some precise sense -- the smallest possible number of vertices.

As our main findings, we present Hopf triangulations of $S^{2k-1}$ for $k \leq 4$, and conjecture that the boundary complex of the $k$-cyclic polytope $\operatorname{kC}(1,2,4,\ldots , 2^{k-1}; n)$, with $n=2^{k+1}-1$ vertices, is a Hopf triangulation of $S^{2k-1}$ (see Conjecture~\ref{conj:kcyclic}). 
In addition, we show that no perfect equilibrium triangulation of $\mathbb{C}P^3$ can exist (Theorem~\ref{thm:noperfectcp3}), and present a nice equilibrium triangulation of real projective $4$-space $\mathbb{R}P^4$ (Theorem~\ref{thm:nicerp4}).

\medskip

Tight embeddings of projective spaces are well known in differential
geometry, but up to now no tight polyhedral $\mathbb{R}P^k$ has been known for $k
\geq 3$. In Section~\ref{sec:tight} we present a tight polyhedral embedding of $\mathbb{R}P^3$ into
$6$-space. Moreover, in Section~\ref{sec:min} we give a simple description of the vertex-minimal
and most symmetric triangulation of $\mathbb{R}P^4$ in terms of a
centrally-symmetric simplicial $5$-polytope.
We conclude with Section~\ref{sec:problems}, where we discuss various open problems.

\subsection*{Acknowledgements}

Research of the second author is supported in part under the Australian Research Council's Discovery funding scheme
(project number DP220102588). This article was completed while the second author was on sabbatical at the 
Technische Universit\"at (TU) Berlin. The authors would like to thank TU Berlin, and particularly
Michael Joswig and his research group, for their hospitality in times of difficult circumstances.

\section{Hopf triangulations of odd-dimensional spheres}
\label{sec:hopf}

We interpret an odd-dimensional sphere $S^{2k-1}$
as the unit sphere in $\mathbb{C}^k$ and -- as its real counterpart --
the unit sphere $S^{k-1}$ in $\mathbb{R}^k$ as the fixed point set of $S^{2k-1}$ under complex conjugation.
By the {\em Hopf fibration} of spheres we mean the fibering by the Hopf action
of $S^1 = \{e^{it}\}$ on $\mathbb{C}^k$ with fibres $\{ e^{it} \cdot x \,\mid\, t \in [0,2\pi ) \}$, $x \in S^{2k-1}$. 
The real analogue is just the real part with fibres defined by the antipodal $S^0 = \{ \pm 1 \}$-action on~$S^{k-1}$.

\medskip

An odd-dimensional sphere $S^{2k-1} = \{(z_1,z_2,\ldots,z_k) \ | \ \sum_i |z_i|^2 = 1\} \subset \mathbb{C}^k$ admits
a natural decomposition, called the {\it Hopf decomposition}, into $k$ parts
\[A_i = \{(z_1,z_2,\ldots,z_k) \in S^{2k-1} \ | \ |z_i| \geq  |z_j| \mbox{ for all } j \}\]
where each $A_i$
is diffeomorphic with a ``solid torus'' $S^1 \times B^{2k-2}$. Naturally,
$S^{2k-1} = A_1 \cup \cdots \cup A_k$ such that the intersection 
$A_{ij} = A_i \cap A_j$
of any pair is diffeomorphic with a ``solid torus'' 
$S^1 \times S^1 \times B^{2k-4}$, the intersection 
$A_{ijl} = A_i \cap A_j \cap A_l$ of 
any triple is diffeomorphic with
$(S^1)^3\times B^{2k-6}$ and so on, until finally $A_1 \cap \cdots \cap A_k$
is the $k$-dimensional {\it central torus} 
\[A_{123 \ldots k} = \{(z_1,z_2,\ldots,z_k) \in S^{2k-1} \ | \ |z_1| = \cdots = |z_k|\} \cong (S^1)^k = T^k.\]

In the case $k=2$, the mapping of $S^3$ onto the total space of its fibres 
recovers the classical {\it Hopf mapping} $S^3 \to S^2$,
where each of the two solid tori $A_1, A_2$ is mapped to a 2-disc (a hemi-sphere of $S^2$),
and the two solid tori fit together along the central $2$-torus 
$S^1\times S^1 = A_{12}= A_1 \cap A_2$ which is mapped onto the
intersection of the two hemispheres (the equator). It follows from the definition
that each part $A_i, A_{ij}, \ldots , A_{123 \ldots k}$ is a union of Hopf fibres.
This motivates the following definition.

\begin{defi}[Hopf triangulation of $S^{2k-1}$]
\label{def:hopf}
A combinatorial triangulation of the sphere $S^{2k-1}$ is called a 
{\em Hopf triangulation}, if all subsets $A_i, A_{ij}, A_{ijl},  \ldots ,
A_{123 \ldots k}$ of its Hopf decomposition are PL homeomorphic with  
subcomplexes of the simplicial complex realising the triangulation.

Here, an abstract simplicial complex is called a {\em combinatorial triangulation
of a (PL) $d$-dimensional manifold}, if the link of every $m$-dimensional
simplex is a triangulated PL standard $(d-m-1)$-sphere.
\end{defi}

The minimum number of vertices for a Hopf triangulation 
of $S^{2k-1}$ is at least the minimum number of vertices which is
required for a $k$-torus $T^k$. This, in turn, is at most $2^{k+1}-1$,
since there are examples of such (combinatorial) triangulations for any $k$ 
\cite{Gr,Ku-La-tori}. It is furthermore conjectured that $2^{k+1} -1$ is also
the smallest possible number of vertices for any combinatorial triangulation of 
$ T^k$, see, for instance \cite[Page 106]{Br-Ku2}. 

\medskip

\noindent
{\bf The real counterpart.} Any sphere $S^{k-1} \subset \mathbb{R}^{k}$
appears as the fixed point set
of complex conjugation of $S^{2k-1} \subset \mathbb{C}^{k}$. We define
a {\it real Hopf decomposition} of $S^{k-1}$ as the real part of the
Hopf decomposition of $S^{2k-1}$ with the same notation otherwise.
Each of the $k$ parts is a ``solid torus'' of type $S^0 \times B^{k-1}$.
The collection of all $k$ parts together can be interpreted as the set
of pairs of antipodal facets of a $k$-dimensional cube, and the
``equilibrium torus'' appears as the set of vertices of this cube.
Consequently the minimum number of vertices of a real Hopf triangulation
of $S^{k-1}$ is $2^k$.

\begin{rem}
It follows from the definition of a Hopf triangulation
that its various parts define a Boolean algebra
which is (anti-)isomorphic with the power set of $\{1,2 \ldots, k\}$
by the assignment $\{i_1,i_2, \ldots i_m\} \mapsto A_{i_1i_2\ldots i_m}$.
Here, the empty set corresponds to the entire sphere.
\end{rem}

\begin{exa} \label{CP2} 
Recall that the $3$-sphere has Hopf decomposition $S^3 = A_1 \cup A_2$ into 
two solid tori $A_1$ and $A_2$ with
a $2$-dimensional torus as the intersection $A_{12} = A_1 \cap A_2$.
A Hopf triangulation with $7$ vertices (actually, the smallest possible number)
is given by the boundary complex of the
{\it cyclic $4$-polytope} $C(7,4)$ \cite{A}, \cite[Theorem 0.7]{Zi} 
which can be defined as
the cyclic $\mathbb{Z}_7$-orbits of the two tetrahedra
$(0 \ 1 \ 3 \ 4)$ and $(0 \ 1 \ 2 \ 3)$ (under $n \mapsto (n+1) \mod 7$) where the second orbit
equals as the first one with vertex labels multiplied by $2 \mod 7$. 

The central torus
is realised as the common boundary of each of them, the unique 
$7$-vertex triangulation represented
by the orbits of $(0 \ 1 \ 3)$ and $(0 \ 2 \ 3)$ modulo $7$,
which is invariant under the multiplication by $2 \mod 7$,
see Figure~\ref{fig:torus}. Note that there is a third solid torus 
of this kind generated by $(0 \ 2 \ 4 \ 6)$, since $2^3 = 1 \mod 7$. 

This leads to three different
Hopf triangulations permuted by multiplication by $2 \mod 7$.
This fact is crucial for the construction of a perfect equilibrium triangulation 
of $\mathbb{C}P^2$, see Proposition~\ref{CP2-equilibrium} and \cite{BK} for details.
\end{exa}

\begin{exa} 
\label{exa:bicyclic}
For any $n = m^2 + m + 1$, there exists a $\mathbb{Z}_n$-symmetric
Hopf triangulation of $S^3$ with $\mathbb{Z}_n$-orbit generators
\[(0 \ 1 \ 2 \ 3), (0 \ 1 \ 3 \ 4), (0 \ 1 \ 4 \ 5), \ldots, (0 \ 1 \ m \ m+1), \]
forming a solid torus, and the same generators after
multiplying them by $m \mod n$ forming the complementary solid torus
to complete $S^3$. 
This triangulation coincides with the boundary complex 
of the {\it bi-cyclic $4$-polytope} of type $\operatorname{2C}(1,m;n)$ \cite{BK,Sm}.
The intermediate torus with generators $(0 \ 1 \ m+1)$ and $(0 \ m \ m+1)$
is one of the chiral tori $\{3,6\}_{(m,1)}$ and also
a special case of the tori considered in \cite{A}. 
The Hopf triangulation
from Example~\ref{CP2} coincides with the case $m=2$. 
\end{exa}

\begin{figure}[bt]
\centering
\epsfig{figure=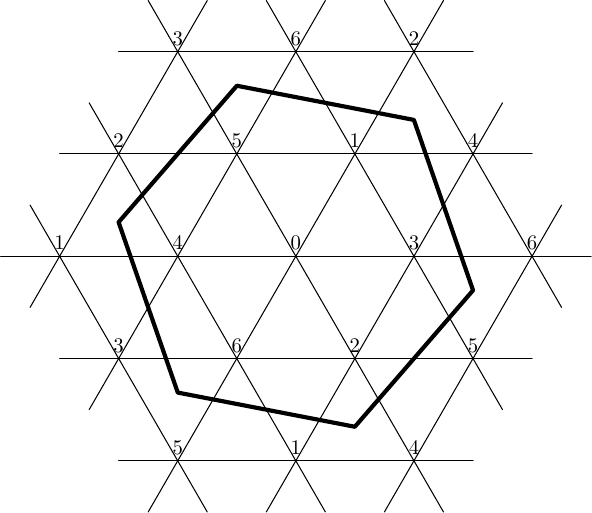,height=50mm}
\caption{The unique 7-vertex triangulation of the torus \label{fig:torus}}
\end{figure}

\noindent
{\bf Triangulations of the central $k$-torus
  and the $(k+1)$-dimensional solid tori.}

For every $k$ there is a very natural candidate for the triangulation of
the central $k$-torus \cite{Ku-La-tori, Ku-La-permcycle}.
One can directly write down all $k! \cdot (2^{k+1}-1)$ top-dimensional
simplices in terms of a single {\it permuted difference cycle}, or {\it permcycle}
for short, see Example~\ref{example:3torus} for notation and some explanations and \cite{Ku-La-permcycle} for a precise definition.

\begin{exa}[3-torus]
\label{example:3torus}
The permcycle 
{\bf 1~2~4~8} is defined as the union of the cyclic $\mathbb{Z}_n$-orbits, $n = 1 + 2 + 4 + 8 = 15$,
of the $3$-simplex with the same labels
and its $3!$ permutations
\[(1  \ 2 \ 4 \ 8), (1 \ 4 \ 2 \ 8), (2 \ 1 \ 4  \ 8), ( 2 \ 4 \ 1 \ 8), (4 \ 1 \ 2 \ 8), (4 \ 2 \ 1 \ 8).\] 

This simplicial complex triangulates the $3$-dimensional torus, as is shown in \cite{Ku-La-permcycle}.
\end{exa}

More generally, a triangulated $k$-torus with $n = 2^{k+1}-1$ vertices is given by the permcycle
\[{\bf 1 \ 2 \ 4 \ 8 \ \ldots \ 2^{k-1} \ 2^k}\] 
generating $k!$ cyclic orbits and hence a total of $k! \cdot n$ simplices of dimension $k$ \cite{Ku-La-tori,Ku-La-permcycle}.
This triangulation is invariant under the following automorphisms:
\begin{eqnarray}
  \tau \colon & x \mapsto x+1 \mod 2^{k+1} - 1 \label{eq:tau} \\
  \rho \colon & x \mapsto -x \mod 2^{k+1} - 1 \label{eq:rho} \\
  \sigma \colon & x \mapsto 2x \mod 2^{k+1} - 1  \label{eq:sigma}
\end{eqnarray}
Equations~(\ref{eq:tau})--(\ref{eq:sigma}) will be referred to frequently throughout the article.
An associated $(k+1)$-dimensional solid torus is given by the permcycle
\[{\bf 1 \ 1 \ 1 \ 4 \ 8 \ \cdots 2^{k-1} \ 2^k},\]
generating a list of \ $n \cdot (k+1)!/3!$ \ simplices, with the above 
triangulation of a $k$-torus as its boundary.
This solid torus is invariant under $\tau$ and $\rho$, but not under $\sigma$. Instead,
$\sigma$ generates an orbit of $k+1$ such solid tori (referred to as its {\em multiples}) -- all necessarily with the
same boundary. They pairwise intersect only in their common boundary (the central $k$-torus). 

\medskip

Geometrically, the central $k$-torus is isometric with a quotient
of the standard lattice triangulation in the hyperplane
$\{x_0 + x_1 + \cdots + x_k = 0\}$ of $(k+1)$-space, such
that $\tau$ corresponds to a translation, $\rho$ corresponds to
a central reflection, and $\sigma$ corresponds to the cyclic shift
of the $k+1$ coordinates \cite{Br-Ku2,Ku-La-tori}.
In each of the solid $(k+1)$-tori, exactly one of the $(k+1)$ coordinate
directions is homotopic to zero. We can assume that introducing
the triangle $(1 \ 2 \ 3)$ corresponds to the direction of the first coordinate.

The fixed point set of $\rho$ consists of $2^k$ isolated points, namely,
the vertex $0$ and the midpoints of all edges $(x\ -x) =(x\ x')$ for 
$x \in \mathbb{Z}_n \setminus \{0\}$ (throughout this article,
we write $x'$ for vertex $-x$ in a simplex).
The fixed point set of one solid $(k+1)$-torus consists of 
a family of $2^{k-1}$ disjoint edges (geometrically: the family
of parallel edges in a $k$-cube). This is interesting in view of relating 
equilibrium triangulations of $\mathbb{C}P^k$ to equilibrium triangulations of $\mathbb{R}P^k$.

\begin{pro}[see \cite{Ku-La-tori}]\label{k-tori}
The above mentioned triangulations of the $k$-torus invariant under $\tau$, $\rho$, and $\sigma$, together with 
$k+1$ associated solid tori of dimension $k+1$ invariant under $\tau$ and $\rho$, and cyclically interchanged by $\sigma$,
exist for every $k \ge 1$.
\end{pro}

Altogether, this triangulated $k$-torus and
($k$ of the $k+1$) solid $(k+1)$-tori are excellent building blocks for a Hopf triangulation of a $(2k-1)$-sphere.

\medskip

\noindent
{\bf Hopf triangulations of the $5$-sphere.}
\label{sec:fivesphere}

If we require a Hopf triangulation of the $5$-dimensional sphere to {\em (a)} contain the $3$-dimensional
$15$-vertex torus permcycle {\bf 1~2~4~8} and {\em (b)} to respect $\tau$ and $\rho$ from 
Equations~(\ref{eq:tau}) and (\ref{eq:rho}), then there are exactly four combinatorial types 
of triangulations of $S^5$ satisfying these requirements (here and in the following we also make frequent reference to automorphism $\sigma$ from Equation~(\ref{eq:sigma})):

\begin{enumerate}
 \item Triangulation $^5 15^7_3$ in \cite{KoLu} with $170$ tetrahedra and symmetry $D_5 \times S_3$ of order $60$. The additional symmetry is $\sigma^2 (^5 15^7_3) =\  ^5 15^7_3$. Each of $^5 15^7_3$ and $\sigma(^5 15^7_3)$ contain two of the four multiples of {\bf 1~1~1~4~8}. Its Hopf decomposition is defined by handlebodies $A_i$, $1 \leq i \leq 3$, with the following generating simplices under $\mathbb{Z}_{15}$-symmetry $x \mapsto (x+1) \mod 15$.
{\small \[ \begin{array}{llll}
( 0\ 1\ 2\ 3\ 4\ 5),& (0\ 1\ 2\ 3\ 5\ 13),& (0\ 1\ 2\ 4\ 5\ 13),& (0\ 1\ 2\ 4\ 12\ 13)\\
( 0\ 1\ 3\ 4\ 7\ 8),& (0\ 1\ 3\ 4\ 8\ 11),& (0\ 1\ 3\ 4\ 11\ 12),& (0\ 1\ 4\ 5\ 8\ 12)\\
(0\ 1\ 3\ 5\ 7\ 8),& (0\ 1\ 3\ 5\ 8\ 13),& (0\ 1\ 3\ 8\ 11\ 13),& (0\ 2\ 5\ 7\ 10\ 12)
\end{array}\]}
 \item Triangulation $^5 15^2_5$ in \cite{KoLu} with $200$ tetrahedra and dihedral symmetry. Each of $\sigma^i(^5 15^2_5)$, $0 \leq i \leq 3$, contain three of the four multiples of {\bf 1~1~1~4~8}. Its decomposition into handlebodies is described in detail below.
 \item Triangulation $^5 15^2_2$ in \cite{KoLu} with $225$ tetrahedra and dihedral symmetry. Each of of $\sigma^i(^5 15^2_5)$, $0 \leq i \leq 3$, contain two of the four multiples of {\bf 1~1~1~4~8}. Its defining partition into handlebodies $A_i$, $1\leq i \leq 3$, is given by the following generators.
{\small \[ \begin{array}{llllll}
(0\ 1\ 2\ 3\ 4\ 5),& (0\ 1\ 2\ 3\ 5\ 6),& (0\ 1\ 2\ 3\ 6\ 7),& (0\ 1\ 2\ 3\ 11\ 12),& (0\ 1\ 2\ 3\ 12\ 13),& (0\ 1\ 3\ 4\ 6\ 7) \\
(0\ 1\ 2\ 3\ 7\ 8),& (0\ 1\ 2\ 3\ 8\ 10),& (0\ 1\ 2\ 3\ 10\ 11),& (0\ 1\ 2\ 7\ 8\ 9),& (0\ 1\ 3\ 8\ 10\ 11)&\\
(0\ 1\ 3\ 4\ 7\ 8),& (0\ 1\ 3\ 4\ 8\ 11),& (0\ 1\ 3\ 4\ 11\ 12),& (0\ 1\ 4\ 5\ 8\ 12)&&
\end{array}\]}
 \item Triangulation $^5 15^7_1$ in \cite{KoLu} with $230$ tetrahedra and symmetry $D_5 \times S_3$. Again, the additional symmetry is $\sigma^2 (^5 15^7_1) =\ ^5 15^7_1$. Each of of $\sigma^i(^5 15^2_5)$, $0 \leq i \leq 1$, contain two of the four multiples of {\bf 1~1~1~4~8}. Its partition into three handlebodies $A_i$, $1 \leq i \leq 3$, is given by the following generators.
{\small \[ \begin{array}{llllll}
(0\ 1\ 2\ 3\ 4\ 5),& (0\ 1\ 2\ 3\ 5\ 6),& (0\ 1\ 2\ 3\ 6\ 7),& (0\ 1\ 2\ 3\ 11\ 12),& (0\ 1\ 2\ 3\ 12\ 13),& (0\ 1\ 3\ 4\ 6\ 7) \\
(0\ 1\ 3\ 4\ 7\ 12),& (0\ 1\ 3\ 7\ 11\ 12),& (0\ 1\ 4\ 5\ 8\ 11),& (0\ 1\ 4\ 5\ 8\ 12),& (0\ 1\ 3\ 5\ 9\ 12),& (0\ 1\ 4\ 5\ 9\ 13)\\
(0\ 1\ 2\ 3\ 7\ 11),& (0\ 1\ 2\ 6\ 7\ 11),& (0\ 1\ 2\ 6\ 10\ 11),& (0\ 1\ 5\ 6\ 10\ 11)&&
\end{array}\]}
\end{enumerate}

Triangulations $3$ and $4$ are polytopal, with combinatorial types the tri-cyclic polytopes $\operatorname{3C}(1,2,4;15)$ and 
$\operatorname{3C}(1,3,4;15)$ respectively. All four Hopf triangulations have a triangulated boundary of the $3$-cube as their fixed point set of $\rho$ (complex conjugation).

Of the four triangulations, only $^5 15^2_5$ has orbit size four under $\sigma$ {\em and} contains three of the four multiples of {\bf 1~1~1~4~8} as $4$-dimensional subsets of its Hopf decomposition. 
All other spheres contain two $4$-dimensional solid tori isomorphic to {\bf 1~1~1~4~8}, as well as one of a different type 
invariant under $\sigma^2$.

Since Triangulation $2$ ($^5 15^2_5$) is the only Hopf triangulation in the above list honouring the setup of central torus and $4$-dimensional subsets described earlier in this section, we focus on the construction of $^5 15^2_5$ as a Hopf triangulation.

\begin{pro}
\label{prop:s5}

There exists a $15$-vertex Hopf triangulation of the $5$-sphere $S^5_{15}$ that contains as $3$-dimensional subset $A_{123}$
the unique $15$-vertex $3$-torus given by permcycle {\bf 1~2~4~8}, and as $4$-dimensional subsets $A_{12}$, $A_{13}$, and $A_{23}$ the solid tori given by the permcycle {\bf 1~1~1~4~8}
and its images under $\sigma$ and $\sigma^2$.
$S^5_{15}$ is equal to $\sigma(^5 15^2_5)$ from \cite{KoLu}.
\end{pro}

\begin{proof}
We start with three $4$-dimensional solid tori of type $B^2 \times \ S^1 \times S^1$ given by permcycle {\bf 1~1~1~4~8} and its images under $\sigma$ and $\sigma^2$, namely {\bf 2~2~2~8~1} and {\bf 4~4~4~1~2} (multiples by $2$ and $4$ modulo $15$). We set these as subsets $A_{32}$, $A_{31}$, and $A_{12}$ of our Hopf triangulation of $S^5$ respectively (cf. Definition~\ref{def:hopf}).

The triangulated 
$A_{23}$ consists of the $\mathbb{Z}_{15}$-orbits of the 
following four 4-simplices
\[(0 \ 1 \ 2 \ 3 \ 7), (0 \ 1 \ 2 \ 6 \ 7), 
(0 \ 1 \ 5 \ 6 \ 7), (0 \ 4 \ 5 \ 6 \ 7),\]
$A_{31} = 2 \cdot A_{23}$ is generated by the four 4-simplices
\[(0 \ 1 \ 3 \ 5 \ 7), (0 \ 2 \ 3 \ 5 \ 7), 
(0 \ 2 \ 4 \ 5 \ 7), (0 \ 2 \ 4 \ 6 \ 7),\]
and $A_{12}$ is generated by
\[(0 \ 1 \ 3 \ 7 \ 11), (0 \ 2 \ 3 \ 7 \ 11), 
(0 \ 1 \ 5 \ 7 \ 11), (0 \ 1 \ 5 \ 9 \ 11).\]

They all share their common boundary -- $3$-dimensional torus given by permcycle {\bf 1~2~4~8}, subset $A_{123}$ -- and nothing else.

The union of 
$A_{23}$ and $A_{31} = 2 \cdot A_{23}$ is the boundary 
of the solid 5-torus $A_{3}$ of
type $B^4 \times S^1$, 
that is given in the left column below. The boundaries of $A_{1}$ and $A_{2}$ are generated analogously. 

Finally the triangulated 5-sphere $S^5_{15}$ consists of the union of 
the triangulated $A_{3}$, $A_{1}$, and $A_{2}$ with the generating
5-simplices in columns (see also \cite{BK}):
\begin{eqnarray*}
(0 \ 1 \ 2 \ 3 \ 5 \ 7)&(0 \ 1 \ 3 \ 5 \ 7 \ 11)&(0 \ 1 \ 2 \ 3 \ 7 \ 11)\\
(0 \ 2 \ 4 \ 5 \ 6 \ 7)&(0 \ 2 \ 4 \ 6 \ 7 \ 11)&(0 \ 1 \ 2 \ 6 \ 10  \ 11)\\
(0 \ 1 \ 2 \ 5 \ 6 \ 7)&(0 \ 2 \ 4 \ 5 \ 7 \ 9)&(0 \ 1 \ 5  \ 9  \ 10 \ 11)\\
(0 \ 1 \ 2 \ 4 \ 5 \ 6)&(0 \ 2 \ 4 \ 7 \ 9 \ 11)&(0 \ 1 \ 5 \ 6 \ 10 \ 11)\\
(0 \ 1 \ 2 \ 3 \ 4 \ 5)&(0 \ 2 \ 4 \ 6  \ 8 \ 10)&  
\end{eqnarray*}
The last column corresponding to $A_{2}$ can be interpreted
as the permcycle {\bf 1~1~1~4~4~4}, but this is not possible
for the others.
This triangulation is invariant under $\tau$ and $\rho$ but
not under $\sigma$.
Instead, the second column is the image of the first one under $\sigma$.
As a consequence, the second column is contained in the intersection of
$S^5_{15}$ and $\sigma(S^5_{15})$ whereas the third column is contained
in the intersection of $S^5_{15}$ and $\sigma^2(S^5_{15})$.

The claim $S^5_{15} = \sigma(^5 15^2_5)$ is straightforward to check, since both this article as
well as \cite{KoLu} work with equally labelled dihedral groups.
\end{proof}

As already pointed out above, this Hopf triangulation is unique under the assumption
of the automorphisms $\tau$, and $\rho$, and the $4$-dimensional handlebodies.

\medskip

\noindent
{\bf A Hopf triangulation of $S^7$ and candidates for arbitrary odd dimension.}
\label{sec:sevensphere}

A Hopf triangulation of the $7$-sphere must contain a central $4$-torus. As above, a canonical candidate for this $4$-torus is the one given by the permcycle ${\bf 1\ 2\ 4\ 8\ 16 }$ -- with the possible minimum of $31$ vertices. Even if we restrict our search to triangulations respecting $\tau$ and $\rho$ (cf. Equations~(\ref{eq:tau}) and (\ref{eq:rho})), a complete classification is computationally not feasible.

Instead, we extrapolate our findings from odd dimensions less than seven to find a suitable candidate. For this, we define the {\em $k$-cyclic polytope} of type $\operatorname{kC}(1, 2, \ldots , 2^{k-1}; n =2^{k+1}-1)$ to be the convex hull of
\[ (e^{2 \pi j/n}, e^{4\pi j/n}, e^{8\pi j/n}, \ldots , e^{2^{k+1}\pi j/n}  ) \in \mathbb{C}^k,  \qquad 0 \leq j \leq n-1.\]
The family of $k$-cyclic polytopes is described in \cite{BK,Sm}, and a facet list for arbitrary parameters can conveniently be constructed in Polymake \cite{polymake} using function \texttt{k\_cyclic()}.

Note that the boundary of the cyclic polytope $\operatorname{1C}(1; 3)$, a triangle, is trivially a Hopf triangulation of $S^1$. The boundary complex of the cyclic polytope $C(4,7) = \operatorname{2C}(1,2; 7)$ is a Hopf triangulation of the $3$-sphere, cf. Examples~\ref{CP2} and \ref{exa:bicyclic}. The boundary of the tri-cyclic polytope $\operatorname{3C}(1,2,4; 15)$ occurs in the list of Hopf triangulations of the $5$-spheres above as $^5 15^2_2$ from \cite{KoLu}. Naturally, this observation leads to the following conjecture.

\begin{conj}
  \label{conj:kcyclic}
  The $n$-vertex boundary complex of the $k$-cyclic polytope $\operatorname{kC}(1, 2, \ldots , 2^{k-1}; n)$, where $n =2^{k+1}-1$, is a Hopf triangulation of the $(2k-1)$-dimensional sphere.
\end{conj}

In Appendix~\ref{app:s7}, we give the decomposition of the boundary complex of $\operatorname{4C}(1,2,4,8 ; 31)$ into four subsets $A_1$, $A_2$, $A_3$, and $A_4$ that define a Hopf decomposition of $S^7$. That is, we show that Conjecture~\ref{conj:kcyclic} is true not just for $k=1,2,3$, but also for $k=4$.

For this, note that the boundary complex of $\operatorname{4C(1,2,4,8 ; 31)}$ decomposes into $353$ orbits under cyclic $\mathbb{Z}_{31}$-symmetry. The entire boundary complex has face-vector 
\[(31, 465, 4340, 21793, 54188, 69130, 43772, 10943 = 31 \cdot 353).\]
The $353$ cyclic orbits partition into $127$ orbits for $A_1$, $100$ orbits for $A_2$, $85$ orbits for $A_3$, and $41$ orbits for $A_4$, each homeomorphic to $B^6 \times S^1$, where $B^6$ can be thought of as the filled-in unit circle in three of the four complex coordinate directions, while the $S^1$ factor is the boundary of the unit circle in the remaining coordinate direction. Note that this decomposition is most like very far from unique.  

Using simpcomp \cite{simpcomp}, or any other software package to handle triangulations of manifolds, it can be checked that the $A_i$, $1\leq i \leq 4$, define a Hopf decomposition of $S^7$, with the common intersection being the central $4$-torus given by permcycle ${\bf 1\ 2\ 4\ 8\ 16 }$.

For finding the precise decomposition into handlebodies $A_i$, $1\leq i \leq 4$, we use the following procedure 
(note that a brute force calculation to find this decomposition is far out of reach computationally).

\begin{enumerate}
  \item For each cyclic orbit representative, compute the coordinates of its barycenters (this is an invariant of the orbit)
  \item For $A_i$ we are looking for a solid torus filling in unit circles in all but the $i$-th complex coordinate direction.
    Hence, we are looking for orbit representatives with barycenters far away from $0$ in absolute value in the $i$-th coordinate direction, 
    but close to $0$ in absolute value in all other coordinate directions
  \item Let $c$ be an orbit representative with barycenters $(b_1,b_2,b_3,b_4) \in \mathbb{C}^4$. Define the {\em rank of $c$} 
    as $\operatorname{rk}_i(c) = \min_{k \neq i} ||b_i|| / ||b_k||$. A value of $\operatorname{rk}_i(c) > 1$ means 
    that the barycenters is further away from $0$ in the $i$-th coordinate than in any other coordinate direction
  \item For a threshold $p > 1$, assign all orbit representatives $c$ with $\operatorname{rk}_i(c) > p$ to $A_i$
  \item Collect all orbit representatives $c$ with $1/p < \operatorname{rk}(c) \leq p$
  \item For every coordinate direction, iteratively combine the orbits from (4) with all subsets of orbits from (5), and check 
    every complex for closedness, Dehn Sommerville equations, and finally its vertex link -- until a bounded manifold is found
  \item Verify that the bounded manifold is of type $S^1 \times B^6$ (using, for instance, discrete Morse theory)
\end{enumerate}

With the correct parameter for $p$ in steps (4) and (5) the above strategy produces a valid partition of the orbits into handlebodies 
$A_i$, $1 \leq i \leq 4$, fairly quickly.
See also Section~\ref{sec:problems} where we discuss open problems related to Conjecture~\ref{conj:kcyclic}.

\section{Equilibrium triangulations of $\mathbb{C}P^k$}
\label{sec:equilcomplex}

First, recall that a (combinatorial) triangulation
of $\mathbb{C}P^k$ requires at least $(k+1)^2$ vertices, and that
for $k \geq 3$ this bound cannot be attained \cite{AM, MY}.
The idea for the proof of this statement is a simple recursion formula: 
$\mathbb{C}P^1$ requires $4$ vertices, and a triangulation
of $\mathbb{C}P^{k+1}$ 
requires a $(2k+2)$-simplex with $2k+3$ vertices
opposite to a triangulated $\mathbb{C}P^k$. Hence, the minimum number $n_k$ of
vertices satisfies the recursion formula
$n_{k+1} = n_k + 2k + 3$, also satisfied by $n_k = (k+1)^2$. 

\medskip

As usual, we describe points in $\mathbb{C}P^k$ in homogeneous
coordinates $[z_0,\ldots,z_k]$. We can decompose the space into $k+1$ 
topological balls $B_0,B_1,\ldots, B_k$ of real dimension $2k$ by defining
\[B_i = \{[z_0,\ldots,z_k] \in \mathbb{C}P^k \ | \ |z_i| \geq  |z_j| \mbox{ for all } j \}.\]
We can think of $p_0 = [1,0,\ldots,0]$ as the centre of $B_0$ and so on.
In particular, $B_0, B_1, \ldots, B_k$ appear as "zones of influence"
of the points $p_0, p_1, \ldots, p_k$.
These points $p_0, \ldots, p_k$ will be vertices of
the equilibrium triangulation. All other vertices will lie in
the central $k$-torus defined by
\[T^k = \{[z_0,\ldots,z_k] \in \mathbb{C}P^k \ | \ |z_0| = |z_1| = \cdots =  |z_k|\} \cong (S^1)^{k+1}/U(1).\]
The various intersections $B_{ij} = B_i \cap B_j, B_{ijk} = B_i \cap B_j \cap B_k$ 
up to $T^k = B_{0123 \ldots k} = B_0 \cap B_1 \cap \cdots \cap B_k$
form -- as in the case of the Hopf decomposition of a sphere --
a Boolean algebra which is (anti-)isomorphic with the
powerset of $\{0,1,2,\ldots, k\}$. The restriction to one of
the $2k$-balls coincides with the structure of a Hopf decomposition
of the boundary of that ball:
\[A_i = B_{0i}, \,\,\,\, A_{ij} = B_{0ij}, \,\,\,\,A_{ijl} = B_{0ijl}, \,\,\,\, \mbox{etc.}\] 

\medskip
Note that the boundary $\partial B_i$ of each of the $B_i$ is a $(2k-1)$-sphere.
Hence, the various parts in the decomposition of $\mathbb{C}P^k$ carry
one more index than in the corresponding Hopf decomposition of $\partial B_i$.
In particular, the central torus in $\mathbb{C}P^k$ has
$k+1$ subscript indices, while the one in the $(2k-1)$-sphere has only $k$ indices.

\begin{defi} ((Perfect) equilibrium triangulation of $\mathbb{C}P^k$)
  
\label{def:equilibrium}
We call a combinatorial triangulation $C$ of $\mathbb{C}P^k$ 
an {\em equilibrium triangulation}, if all subsets $B_i, B_{ij}, B_{ijl},  
\ldots ,B_{0123 \ldots k}$ are PL homeomorphic with 
subcomplexes of $C$.

$C$ is called 
{\em perfect}, if its central torus is the $(2^{k+1} -1)$-vertex triangulation  
generated by permcycle ${\bf 1 \ 2 \ 4 \ \cdots 2^{k-1} \ 2^k}$,
$C$ has exactly one additional vertex $p_i$, per subset $B_i$, $0 \leq i \leq k$,
and the automorphisms $\tau$, $\rho$, and $\sigma$, as defined in 
Equations~\ref{eq:tau}-\ref{eq:sigma}, extend to $C$ with $\tau(p_i) = p_i$, 
$\rho(p_i) = p_i$, and $\sigma(p_i) = p_{(i+1) \mod (k+1)}$ for all $0 \leq i \leq k$.
\end{defi}

\begin{exa}
In the trivial case $k=1$, the central $1$-torus is the $1$-sphere
$\{[z_0,z_1] \in \mathbb{C}P^1 \ \big| \ |z_0| = |z_1|\}$
and can be triangulated by three vertices $0$, $1$, $2$ and
the orbit of the edge $(0 \ 1)$ under
the action of $\mathbb{Z}_3$. By adding the two centres $p_0$ and $p_1$
we obtain the double cone over a triangle. The triangulation is invariant under
the involution $x \mapsto -x \mod 3$ and under the cyclic shift
$(1 \ 2) (0) (p_0 \ p_1)$, and we have a perfect equilibrium triangulation 
of $\mathbb{C}P^1 \cong S^2$. See Figure~\ref{fig:CP1} for a picture.

\begin{figure}
  \includegraphics[height=5cm]{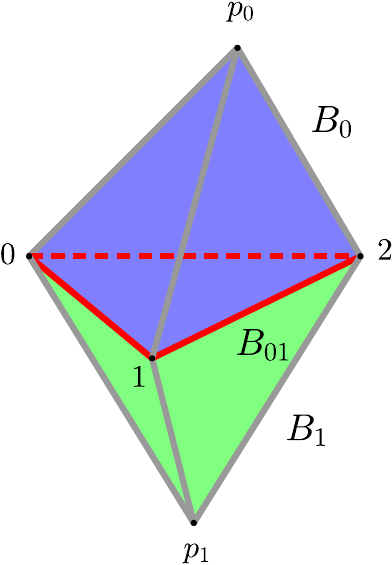}
  \caption{The perfect equilibrium triangulation of $\mathbb{C}P^1$. \label{fig:CP1}}
\end{figure}
\end{exa}

\begin{pro} \label{CP2-equilibrium}
There exists a perfect equilibrium triangulation of $\mathbb{C}P^2$.
\end{pro}

The proof is given in \cite{BK}. The construction is based on
the Hopf triangulation of $S^3$ with $7$ vertices from Example~\ref{CP2}:
The central torus $T^2$ -- defined as the $\mathbb{Z}_7$-orbits 
of the two triangles $(0 \ 1 \ 3)$  and $(0 \ 2 \ 3)$ --
is triangulated with the minimum number of 7 vertices. 
This is invariant under the actions of $\tau$, $\rho$, and $\sigma$
generating a group of order 42, see Figure~\ref{fig:torus}.
The orbit of the tetrahedron $(0 \ 1 \ 2 \ 3)$ 
defines a solid torus with $7$ vertices which is not invariant under
$\sigma$ (multiplication by $2 \mod 7$). Since $2^3 = 1 \mod 7$ this defines
two other triangulated solid tori with the same boundary where the union of any
two of them is a triangulated $3$-sphere. In fact, it is the
boundary complex of the cyclic polytope $C(4,7)$ in three versions.
Finally, we take the cones from the three special vertices $p_0$, $p_1$, and $p_2$
to these three triangulated $3$-spheres. This defines three triangulated 
$4$-balls intersecting mutually in the three solid tori. The intersection 
of all three recovers the central $7$-vertex torus.
This equilibrium triangulation of $\mathbb{C}P^2$
has $10$ vertices and $42$ top-dimensional simplices.
A small calculation reveals that the extensions of $\tau$, $\rho$ 
(both with $p_0$, $p_1$, and $p_2$ as fixed points), and $\sigma$ (with \
cyclic shift $p_0 \mapsto p_1 \mapsto p_2 \mapsto p_0$) to the triangulation are in fact automorphisms
and hence it is perfect. 

Another very natural 10-vertex triangulation of $\mathbb{C}P^2$
was obtained by Bagchi and Datta \cite{BD1}.
A similar decomposition of a manifold into three 4-balls is also essential in \cite{Sch} -- a 
special case of what is called a {\it trisection} in the topology of 4-manifolds \cite{GK}.

\begin{rem}
  If we start with a different bi-cyclic polytope $\operatorname{2C}(1,m;n)$,
  $m>2, n = m^2 + m + 1$, as the Hopf triangulation of $S^3$,
  we obtain a different (non-perfect) equilibrium triangulation of $\mathbb{C}P^2$. The automorphism $\sigma$ in these triangulations is given by
  multiplication by $m \mod n$.
\end{rem}

The real part of the perfect triangulation of $\mathbb{C}P^2$ is the 7-vertex
triangulation of $\mathbb{R}P^2$ obtained by antipodal identification
of the boundary of the so-called {\it tetrakis hexahedron}:
an ordinary 3-cube where each square is subdivided by an extra vertex
at the centre, see \cite{BK} and Figure~\ref{fig:RPk}.

\medskip

Assuming that we have a suitable Hopf triangulation of $S^{2k-1}$ with
$2^{k+1} - 1$ vertices,
we can hope for a perfect equilibrium triangulation of $\mathbb{C}P^k$ by
coning it to $p_0$ and taking the union of all images under the permutation
\[\sigma \colon x \mapsto 2x \mod 2^{k+1} - 1,\]
extended to vertices $p_0, \ldots, p_k$ by the shift
$\sigma \colon p_0 \mapsto p_1 \mapsto \cdots \mapsto p_k \mapsto p_0$
of order $k+1$.

\medskip

\begin{theorem}[Main theorem on perfect equilibrium triangulations] \label{thm:main}

Assume that we have a Hopf triangulation $S$ of $S^{2k-1}$ with $2^{k+1}-1$
vertices containing as central $k$-torus the permcycle ${\bf 1 \ 2 \ 4 \ \cdots 2^{k-1} \ 2^k}$
together with $k$ of its $(k+1)$-dimensional solid torus fillings as described in Proposition~\ref{k-tori}.
Assume further that for each 
intermediate dimension the subsets representing solid tori 
and their images under $\sigma$ (when extended to the
entire Hopf triangulation of the sphere) are relatively disjoint, meaning
that they overlap only along their mutual boundaries.

\smallskip
Then we obtain a perfect equilibrium triangulation of
$\mathbb{C}P^k$ by adding a cone from the point $p_0 = [1, 0, 0, \ldots, 0]$ to 
the original Hopf triangulation and applying all powers
of $\sigma$ to this $2k$-ball, where $p_i$ is the $i^{th}$
cyclic shift of $p_0$ and $\sigma(p_i) = p_{(i+1) mod (k+1)}$ cyclically.
This equilibrium triangulation with $2^{k+1}+k$ vertices 
is defined as the union of these $k+1$
triangulated $2k$-balls, and they all intersect precisely
in the central $k$-torus.
\end{theorem}

\begin{rem}
\label{rem:sigmainv}
If the assumption on $k+1$ disjoint $2k$-balls is not
satisfied, and hence a perfect equilibrium triangulation cannot be constructed 
directly, then a subdivision of the triangulation can recover a (non-perfect) 
equilibrium triangulation of $\mathbb{C}P^k$, see the example for $k=3$ 
below. There, for a given
Hopf triangulation as the boundary of one of the balls, the images under
$\sigma$ are not disjoint because there are ``double faces'' that
are (possibly as sets of faces of one of the $2k$-balls) invariant under $\sigma$ and -- therefore -- occur in multiple of the
$2k$-balls $A_i, A_j$, but should not occur in the triangulated
intersection $A_{ij}$.
\end{rem}

\begin{proof}[Proof of Theorem~\ref{thm:main}]
Let us recall the geometry of the central equilibrium torus
\[T^k = \{[z_0,\ldots,z_k] \in \mathbb{C}P^k \ | \ |z_0| = |z_1| = \cdots =  |z_k|\}\]
in the standard Fubini-Study metric.
Multiplying a single coordinate by $e^{i\theta}$ is an isometry,
hence the equilibrium torus is homogeneous and therefore flat.
All permutations of the $k+1$ coordinates are isometries as well. 
We can identify a point $[e^{it_0}, \ldots, e^{it_k}]$
in homogeneous coordinates with the point
$(t_0, t_1, \ldots, t_k)$ in the (flat) hyperplane of $\mathbb{R}^{k+1}$
where the sum of all coordinates is zero.
This hyperplane appears as the universal covering of the central 
equilibrium torus.
We know from \cite{Ku-La-tori} that the universal covering
of the triangulated
$k$-torus from Proposition~\ref{k-tori} admits a geometric
realisation in this (flat) hyperplane such that the automorphism
$\sigma$ corresponds
to the cyclic shift of the $k+1$ coordinates.
Therefore, we can identify the triangulated $k$-torus with the
central equilibrium torus of $\mathbb{C}P^k$.

\medskip
Furthermore, the Hopf triangulation $S$ contains $k$ of the $(k+1)$-dimensional 
solid tori. One of them
introduces the triangle $(1 \ 2 \ 3)$ with the effect that one
particular coordinate direction becomes null homotopic, again due to Proposition~\ref{k-tori}.
By the automorphism $\sigma$, the other solid $(k+1)$-tori
correspond to the shifted coordinate directions.
This is the same in $\mathbb{C}P^k$.
Thus, we can identify the $(k+1)$-dimensional solid tori with
the subset $B_{123\cdots k}$ and their cyclic shifts.
This procedure can be extended to the higher-dimensional solid tori:
Each of them is a part of the Hopf triangulation which is the link
of the vertex $[1,0,\ldots,0]$. The automorphism $\sigma$, with 
$\sigma(p_i) = p_{i+1}$, cyclically shifts them to the entire projective space:
\[A_i = B_{0i}, \,\,\,\, A_{ij} = B_{0ij}, \,\,\,\, A_{ijl} = B_{0ijl}, \,\,\,\, \mbox{etc.}\] 
By construction, the ball $B_0$ is the cone from $p_0$ to the
given Hopf triangulation $S$, and $\sigma(B_i) = B_{i+1}$ cyclically.
Here, we use the assumption
that the various parts of the triangulated
solid tori intersect only along their common boundaries:
For each intermediate dimension, their relative interiors 
must be disjoint, because this is the case in the actual $\mathbb{C}P^k$.
This implies that the vertex star of one of the vertices
in the central $k$-torus is a triangulated $2k$-ball, decomposed
into the $k+1$ parts lying in the various $B_i$.
The link of such a vertex is decomposed into $k+1$ PL $(2k-1)$-balls
(with mutually disjoint relative interiors)
according to the decomposition of $\mathbb{C}P^k$
into $2k$-balls $B_0, \ldots, B_k$.
By assumption, the triangulations of these relative interiors are also
mutually disjoint.
\end{proof}

{\bf A (non-perfect) equilibrium triangulation of $\mathbb{C}P^3$.}

We want to apply Theorem~\ref{thm:main} to the Hopf triangulation $S^5_{15}$
from Section~\ref{sec:fivesphere}. The central $3$-torus contains all 15 vertices and all
$15 \choose 2$ edges,
and the automorphism $\sigma$ of order $4$ indeed represents
the cyclic shift of the coordinates in the hyperplane
of $4$-space defined by vanishing of the sum of the $4$ coordinates 
\cite{Ku-La-tori}, similarly for the $4$-dimensional solid tori. 
Hence, the first part of the assumptions
in Theorem~\ref{thm:main} is satisfied.
However, the issue of non-disjointedness of
the various parts leads to difficulties in the
construction of an equilibrium triangulation of $\mathbb{C}P^3$
from this particular Hopf triangulation of $S^5$.
If we start with the cone from $p_0$ to $S^5_{15}$ and apply 
the permutation $\sigma$, we obtain four 6-balls as required.
However, this is not quite a combinatorial triangulation
for the following reasons:

\begin{enumerate}
  \item The (short) $\mathbb{Z}_{15}$-orbit of $(0\  5\  10) = (0\  5\  5')$ is invariant under $\sigma$ and hence contained in each
of the four $6$-balls. However, it is not contained in the central torus. These triangles have disconnected vertex links in the union of the four $6$-balls.
  \item The $\mathbb{Z}_{15}$-orbits of $(3\  5\  5'\ 3')$ and $\sigma((3\  5\  5'\ 3')) = (5\  6\  6'\ 5')$ are invariant under $\sigma^2$ and contained 
  in each of the four 6-balls but not in the central torus. These quadruples have disconnected links in the union of the four $6$-balls as well. Note
  that triangles $(3\  5\ 3')$ and $(3\  5'\ 3')$, and $(5\  6\  6')$ and $(6\  6'\ 5')$ in the boundaries of $(3\  5\  5'\ 3')$ 
  and $(5\  6\  6'\ 5')$ respectively are not contained in the central torus either, while the other boundary triangles are.  
\end{enumerate}

In principle one can repair this problem by introducing a subdivision of those 
simplices which have disconnected links in the union of the four 6-balls:

\begin{enumerate}
  \item All five triangles in the orbit of $(0\ 5\ 5')$ have as their links
  two disjoint $3$-spheres. Stellarly subdividing triangles in
  simplices contributing to one of these $3$-spheres fixes this problem.
  This step requires five additional vertices in total.
  \item The case of the orbits of $(3\  5\  5'\ 3')$ and $(5\  6\  6'\ 5')$ is somewhat 
  more difficult: each of the $30$ tetrahedra in these orbits has as their link
  two disjoint $2$-spheres. Subdividing tetrahedra in simplices of one of these $2$-spheres
  as above leaves us with triangle orbits of $(3\  5\ 3')$, $(3\  5'\ 3')$, $(5\  6\  6')$, 
  and $(6\  6'\ 5')$ still in the union of all four $6$-balls, but not in the central torus.
  Instead, we subdivide each of the triangles in these four orbits (following the procedure described above) 
  and choose a triangulation of the respective surrounding tetrahedra into five tetrahedra each,
  with the added vertices in the boundary triangles spanning an edge.
  This step simultaneously takes care of the bad tetrahedra links and requires a
  total of $4 \cdot 15 = 60$ additional vertices.
\end{enumerate}

Altogether this yields the following result:

\begin{pro}
There exists an equilibrium triangulation of $\mathbb{C}P^3$ with $84$ vertices.
\end{pro}

\begin{proof}
  This is a corollary of Theorem~\ref{thm:main}.
  We start with the complex obtained from applying $\sigma$ to the cone of 
  $S_{15}^5$ over $p_0$.

  As already mentioned, the first part of the assumptions of Theorem~\ref{thm:main}
  is already satisfied. Following Remark~\ref{rem:sigmainv} it remains to check the complex 
  for the non-disjointness conditions, that is, for sets of faces that are invariant under $\sigma$,
  but not in the central $3$-torus.
    
  Since the central $3$-torus already contains all ${15 \choose 2} = 105$ 
  possible edges of the vertex set of the intersection of the four $6$-balls,
  we only need to check links of faces of dimension at least two. 
  As a result of the subdivisions described above, the links of all triangles
  are valid, as are the links of all tetrahedra. The links of $4$- and $5$-faces
  are already valid in the initial complex.
  
  The result follows from the fact that the initial complex has $15+4=19$ vertices
  and the subdivisions add another $5 + 60 = 65$ vertices.
\end{proof}

Unfortunately, our triangulation of $\mathbb{C}P^3$ has considerably more vertices than what is needed
for the construction of $\mathbb{C}P^3$ in \cite{BD2}. The theoretical lower bound is $17$ vertices \cite{AM}.

\medskip

We can strengthen the observation that the Hopf triangulation of the $5$-sphere 
from Proposition~\ref{prop:s5} does not produce a perfect equilibrium 
triangulation of $\mathbb{C}P^3$ as follows. 

\begin{thm} 
  \label{thm:noperfectcp3}
  There does not exist a perfect equilibrium triangulation of $\mathbb{C}P^3$.
\end{thm}

\begin{proof}
  By definition, a perfect equilibrium triangulation of $\mathbb{C}P^3$ can be constructed from a $15$-vertex Hopf triangulation of the $5$-sphere respecting $\tau$ and $\rho$, and containing the $3$-dimensional torus permcycle {\bf 1~2~4~8}. A complete list of such triangulations is $^5 15^7_3$, $^5 15^2_5$, $^5 15^2_2$, and $^5 15^7_1$ in \cite{KoLu} as detailed in Section~\ref{sec:fivesphere}. Of these, only $^5 15^2_5 \cong S^5_{15}$ 
produces exactly four distinct $4$-dimensional subsets $B_{012}$, $B_{013}$, $B_{023}$, and $B_{123}$ by applying $\sigma$. To see this check the $4$-dimensional solid tori contained in all Hopf triangulations of the $5$-sphere listed in Section~\ref{sec:fivesphere}. On the other hand, $S^5_{15}$ is not suitable either by the discussion above.
\end{proof}

\begin{rem}
  Note that the Hopf triangulation of $S^7$ as presented in Section~\ref{sec:hopf} and Appendix~\ref{app:s7} is not suitable to produce a 
  perfect equilibrium triangulation of $\mathbb{C}P^4$. To see this note that handlebodies $A_i$, $1\leq i \leq 4$ all have distinct sizes.
\end{rem}

\section{The real case: Equilibrium triangulations of $\mathbb{R} P^k$}
\label{sec:equilreal}

In the case of real spheres $S^{k-1}$ of arbitrary dimension 
and real projective spaces $\mathbb{R}P^k$, we have an equivalent decomposition
as in the complex case from Section~\ref{sec:equilcomplex} by considering the real part of
$\mathbb{C}P^k$. We simply replace
the complex coordinates $z_0, \ldots,z_k$ by real coordinates $x_0,\ldots,x_k$,
and the $k$-torus $(S^1)^k$ by the ``$0$-dimensional $k$-torus'' $(S^0)^k$ with
$2^k$ isolated points.
Then the 1-dimensional ``solid torus'' appears as 
$B^1 \times S^0 \times S^0 \times \cdots \times S^0$ 
which is nothing but a collection
of $2^{k-1}$ disjoint intervals. A good picture of such
a real Hopf decomposition of $S^{k-1}$ appears
 if we interpret it as the boundary of
a $k$-cube with $2^k$ vertices, $2^{k-1}$ parallel edges in one
direction, $2^{k-2}$ parallel squares in two directions and so on.
Finally, the boundary of the cube is represented as the union of
$k$ pairs of opposite facets, each topologically of type $B^{k-1} \times S^0$.

\begin{defi}[Equilibrium triangulation of $\mathbb{R}P^k$]
  
We call a combinatorial triangulation $C$ of $\mathbb{R}P^k$ 
an {\em equilibrium triangulation}, if all subsets $B_i, B_{ij}, B_{ijl},  
\ldots ,B_{0123 \ldots k}$ (with real coordinates $x_i$ instead of
complex coordinates $z_i$) are PL homeomorphic with 
subcomplexes of $C$.
In this case the central ``torus'' $T^k = B_0 \cap \cdots \cap  B_k$
consists of $2^k$ isolated points. 
\end{defi}

In view of the complex case, it makes sense to start with
an equilibrium triangulation
of $\mathbb{C}P^k$ and to regard $\mathbb{R}P^k$ as
the fixed point set of the involution $\rho \colon x \mapsto -x$. 
In the (trivial) case of $\mathbb{C}P^1$, we obtain a circle divided into
four edges, where the four vertices are $0$, $p_0$, $p_1$ and
the centre of the edge $(1 \ 2)$, see Figure~\ref{fig:RPk} on the left. In the case 
of the equilibrium triangulation of $\mathbb{C}P^2$ with $7 + 3$
vertices, this fixed point set is an equilibrium triangulation 
of $\mathbb{R}P^2$ with $4 + 3$ vertices, see \cite{BK} and Figure~\ref{fig:RPk} on the right.
The four fixed points of the central torus are $0$ and the midpoints of
the edges $(1\ 6)$, $(2\ 5)$, $(3\ 4)$, the three exterior vertices are the
same as in the complex case $p_0 = [1,0,0]$, $p_1 = [0,1,0]$, and $p_2 = [0,0,1]$.
The automorphism group contains the cyclic shift 
$p_0 \mapsto p_1 \mapsto p_2 \mapsto p_0$,
combined with
the multiplication by $2 \mod 7$ on the pairs $\pm x$ of vertices in the
central torus, a clockwise rotation by $2\pi / 3$ in Figure~\ref{fig:RPk} on the right.

\begin{figure}
  \includegraphics[width=0.8\textwidth]{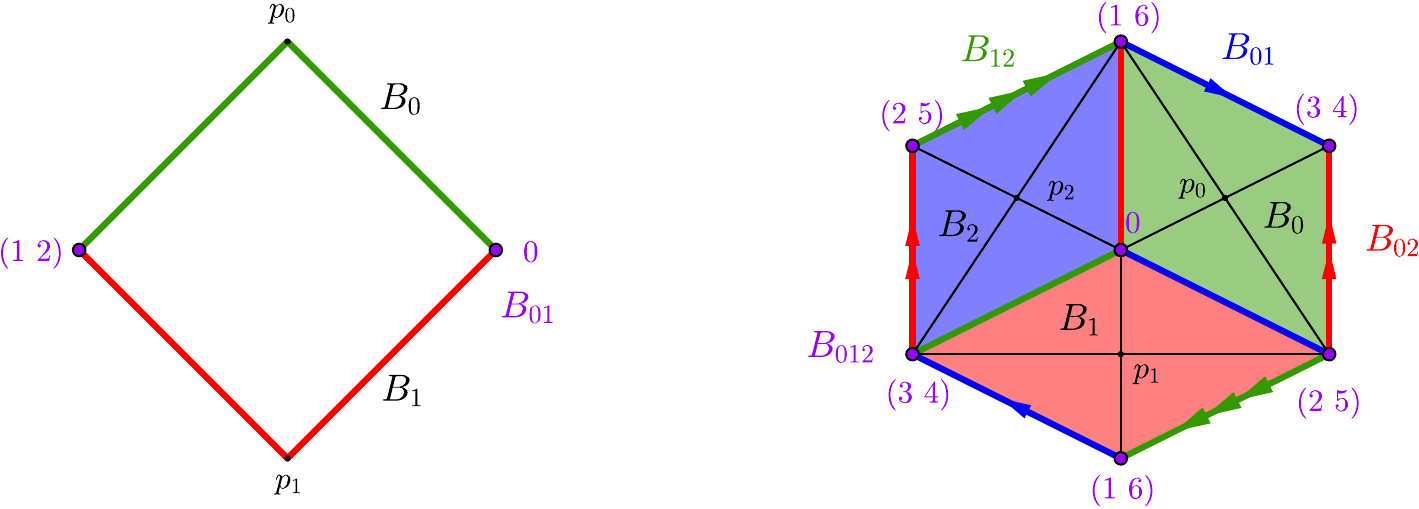}
  \caption{Equilibrium triangulations of (a) $\mathbb{R}P^1$ and (b) $\mathbb{R}P^2$ as real parts of perfect equilibrium triangulations of $\mathbb{C}P^1$ and $\mathbb{C}P^2$.
  Both triangulations are, in fact, nice equilibrium triangulations of real projective spaces, cf. Definition~\ref{def:nice}. \label{fig:RPk}}
\end{figure}

\begin{pro}
Given a perfect equilibrium triangulation of $\mathbb{C}P^k$
with automorphisms $\tau$, $\rho$, and $\sigma$, 
there exists an equilibrium triangulation of $\mathbb{R}P^k$ with $2^k$
vertices defined as the fixed point set of the involution $\rho$.
Its vertices are $0$ and the midpoints of the edges $(x\ x')$ for each
$x \in \mathbb{Z}_n \setminus \{0\}$.
\end{pro}

\begin{proof}
Under the automorphisms $\tau$ and $\rho$
the outer vertices $p_0, \ldots, p_k$ are fixed.
Since $\rho$ acts as complex conjugation on the central
equilibrium torus, this extends naturally to the higher dimensional
solid tori since it is determined by $\rho(x) = -x$ mod $n$ which
admits one and only one simplexwise linear extension.
From Proposition~\ref{k-tori} one can even determine
the positions of all vertices of the central $k$-torus in terms
of homogeneous coordinates: In the additive structure of
the hyperplane of all $(t_0, \ldots, t_k)$ with $\sum_jt_j = 0$
the automorphism $\rho$
acts as $t_j \mapsto -t_j$, in the complex description in homogeneous
coordinates with $[e^{it_0}, \ldots, e^{it_k}]$ it appears as the explicit
conjugation $e^{it} \mapsto e^{-it}$.
\end{proof}

\medskip

\noindent
    {\bf $\mathbb{R}P^3$ as the fixed point set of $\rho$ in
      the triangulated $\mathbb{C}P^3$.}

For $\mathbb{R}P^3$ we can determine the fixed point
set of the Hopf triangulation of the 5-sphere in Section~\ref{sec:fivesphere} under the involution $\rho : x \mapsto (-x \mod 15)$.
The fixed point set contains $0$ and the centres of the seven edges $(x\ x')$
which we denote by $\pm x$ or even by $x$ if there is no danger of confusion.
From the original 5-sphere we obtain the triangulated 2-sphere 
as the fixed point set of $\rho$ with triangles
\[(0\ 1\ 2), (1\ 2\ 3), (4\ 5\ 6), (5\ 6\ 7), \ \ (0\ 2\ 4), (2\ 4\ 6), (1\ 3\ 5), (3\ 5\ 7), \ \ (0\ 1\ 5), (0\ 4\ 5), (2\ 3\ 7), (2\ 6\ 7)\]
where the three blocks come from the three columns (representing solid tori)
in Section~\ref{sec:fivesphere} in the same order.
Each block represents a ``solid torus" of type $S^0 \times B^2$,
in fact nothing but a pair of opposite facets of a $3$-cube
triangulated by the main diagonal $(2\ 5)$ (which, however, is not part
of the triangulation). For the equilibrium triangulation of $\mathbb{R}P^3$
we put one of the outer vertices into the centre of this triangulated
2-sphere. Then we apply the cyclic permutation $p_0 \mapsto p_1 \mapsto p_2 \mapsto p_3 \mapsto p_0$
combined with $\tilde{\sigma} = (1247)(36)(0)(5)$ arising from the
multiplication by $2 \mod 15$ ($\sigma$ from Equation~\ref{eq:sigma}) after 
identification $\pm$ under complex conjugation ($\rho : x \mapsto (-x \mod 15)$ from Equation~(\ref{eq:rho})).
This leads to four solid cubes covering $\mathbb{R}P^3$.

\medskip
However, as in the case of $\mathbb{C}P^3$, this does not quite 
give a combinatorial triangulation since the links of the edges $(0\ 5)$, 
$(3\ 5)$ and $(6\ 5)$ consist of two cycles instead of one. 
Again, as in the case of $\mathbb{C}P^3$, we can repair this problem by subdividing faces.
This leads to a $15$-vertex equilibrium triangulation of $\mathbb{R}P^3$.

Note that this $15$-vertex triangulation of $\mathbb{R}P^3$ is isomorphic
to the triangulation that is the fixed point set under complex conjugation of 
the $84$-vertex equilibrium triangulation of $\mathbb{C}P^3$. For this we assume
that its added vertex in the centre of triangle $(0\ 5\ 5')$ lies in the fixed point set
as the added vertex on edge $(0\ 5)$.
We furthermore assume that the fixed point on the edge connecting the two extra vertices 
added to the centres of $(3\ 5\ 3')$ and $(3\ 5'\ 3')$ in tetrahedron 
$(3\ 5\ 5'\ 3')$ coincides with the extra vertex subdividing $(3\ 5)$ in the triangulated 
equilibrium $\mathbb{R}P^3$, and similarly for $(5\ 6\ 6'\ 5')$ and the edge $(5\ 6)$.

\medskip

\noindent
{\bf Other equilibrium triangulations of $\mathbb{R}P^3$.}

Independent of considerations about complex projective spaces
there is the following direct approach
for obtaining an equilibrium triangulation of $\mathbb{R}P^k$:

We regard $\mathbb{R}P^k$ as the boundary of a $(k+1)$-cube
after identification of antipodal points.
These $k+1$ cubes define a cubical decomposition which, unfortunately,
is not polyhedral because its $k$-dimensional cubes pairwise intersect in a disconnected set.
However, the following statement still holds.

\begin{pro}
  \label{prop:polyhedral}
  There exists a polyhedral decomposition of $\mathbb{R}P^k$ with $2^k + k + 1$ vertices and into 
  $2k(k+1)$ pyramids over the $(k-1)$-cube containing all subsets of an equilibrium 
  decomposition of $\mathbb{R}P^k$ as polyhedral subcomplexes.
\end{pro}

\begin{proof}
  We start with the cubical decomposition of $\mathbb{R}P^k$ as boundary of a $(k+1)$-cube
after identification of antipodal points. For every boundary of one of the $(k+1)$ cubes of 
dimension $k$, we take its $2k$ cubical facets of dimensions $k-1$ and cone them to 
an extra vertex at the centre of the $k$-cube.

In total, this produces a decomposition of $\mathbb{R}P^k$ into $2k(k+1)$ pyramids over the
$(k-1)$-cube with $2^k + k + 1$ vertices. The decomposition is polyhedral, because connected 
components of intersections of pairs of $k$-cubes lie in different pyramids: any two 
$(k-1)$-cubes of the cubical decomposition intersect at most in a single common cube of some dimension.
\end{proof}

A straightforward approach to obtain an equilibrium triangulation of $\mathbb{R}P^k$ is thus to 
triangulate the $(k-1)$-cubes in the polyhedral complex from Proposition~\ref{prop:polyhedral}. 
This motivates the following definition.

\begin{defi}
\label{def:nice}
Given the polyhedral decomposition $C$ of $\mathbb{R}P^k$ with $n = 2^k + k + 1$ vertices and into $2k(k+1)$ pyramids, we say that an equilibrium triangulation of $\mathbb{R}P^k$ is {\em nice}, if it is a simplicial subdivision of $C$ without additional vertices.
\end{defi}

The existence of such a triangulation is trivial for $k=1$ and easy for $k=2$:
Half the 1-skeleton of a 3-cube is already simplicial
as the complete graph $K_4$ with vertices $0$, $1$, $2$, $3$,
and by inserting three extra vertices into the three squares
$(0\ 1\ 2\ 3)$ and their images under $\sigma = (123)(0)$
we obtain a $7$-vertex nice equilibrium triangulation of the
real projective plane. It coincides with the real part of the perfect 
equilibrium triangulation of $\mathbb{C}P^2$ from Proposition~\ref{CP2-equilibrium},
for a picture see \cite{BK} or \cite{Sch}. Accordingly, Figure~\ref{fig:RPk}
shows nice equilibrium triangulations of $\mathbb{R}P^k$, $k=1,2$.

\bigskip

\begin{figure}[htb]
\includegraphics[width=.6\textwidth]{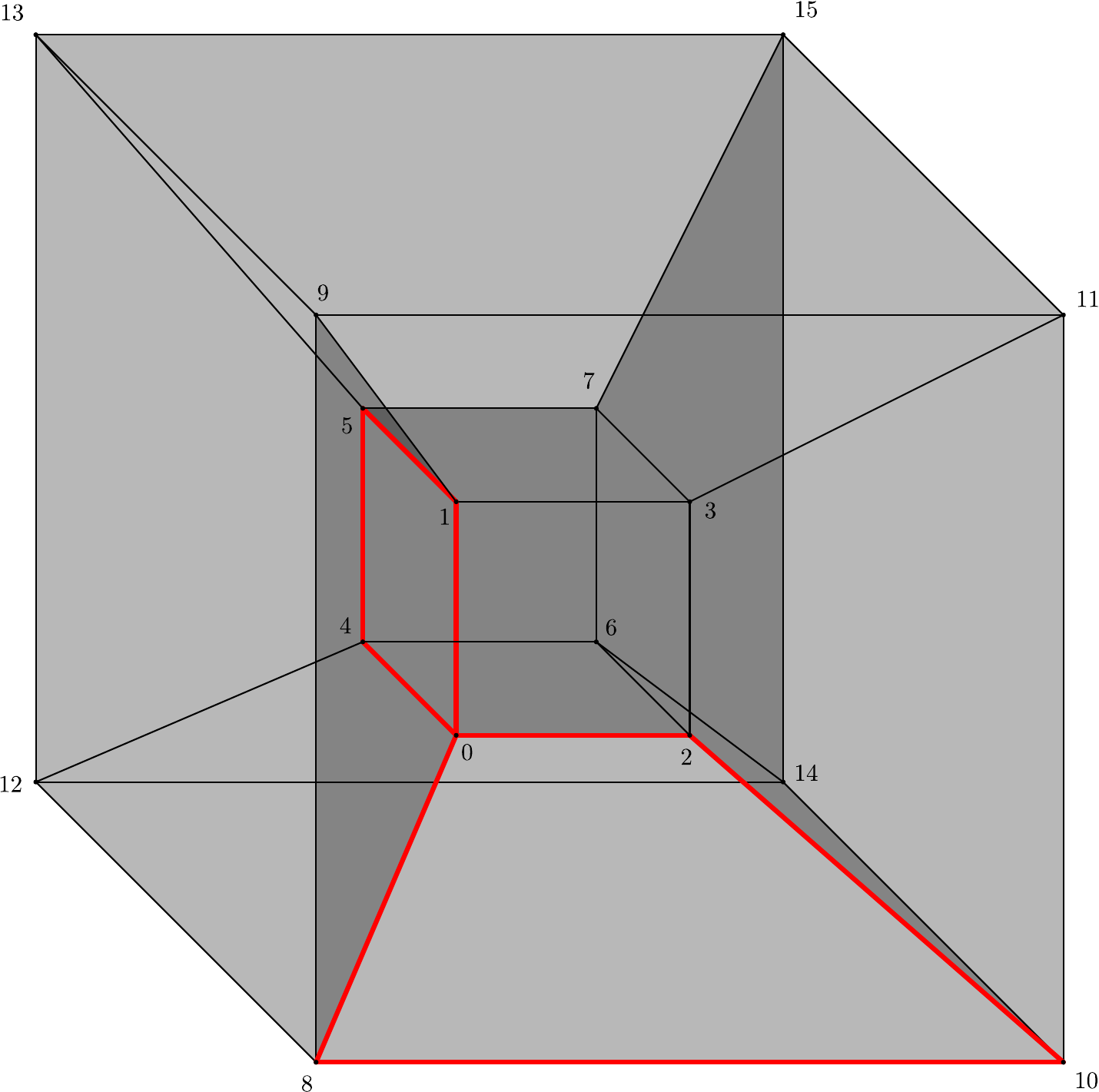}
\caption{The boundary of the $4$-cube and the ``Clifford torus''. The generators of homology for the torus are indicated in red. \label{fig:fourcube}}
\end{figure}

For $k=3$ we have the following explicit construction:

\begin{pro}
\label{prop:equilibriumRP3}
There exists a nice equilibrium triangulation of $\mathbb{R}P^3$ (with $12$ vertices).  
\end{pro}
  
\begin{proof}
We start with a cubical Hopf decomposition of the boundary complex
of the $4$-cube $[0,1]^4$ with vertices labelled by their coordinates written in base $10$.
It is based on the representation as a Cartesian product of two squares.
There are two solid tori, each consisting of four $3$-cubes, with a
common boundary, a torus decomposed into $16$ squares by a $(4\times 4)$-grid.
This ``cubical Clifford torus'' in the boundary of the $4$-cube
is depicted in Figure~\ref{fig:fourcube}, and in Figure~\ref{fig:grid} on the left.

\begin{figure}[hbt]
\includegraphics[width=\textwidth]{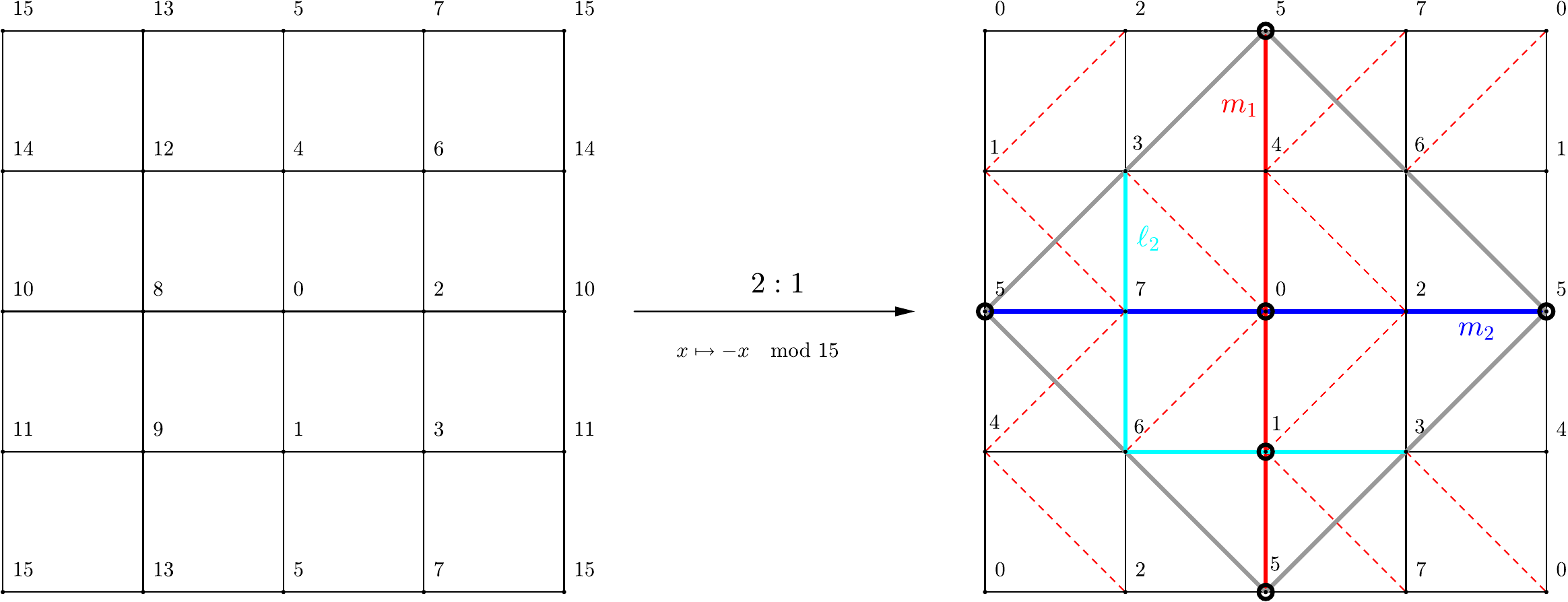}
\caption{Left: Clifford torus in the boundary of the $4$-cube drawn as a $(4\times 4)$-grid in the plane. Right: image of the torus under identification $x \mapsto (-x \mod 15)$ defining a $2:1$ covering. A fundamental domain is indicated in grey. The $8$-square torus is invariant under $\tilde{\sigma} = (1247)(36)(0)(5)$. In red: meridian curve of solid torus on one side of the Clifford torus. In blue and cyan: framing of the solid torus on the other side.\label{fig:grid}}
\end{figure}

We obtain an equilibrium triangulation of $\mathbb{R}P^3$ by antipodal
identification of the boundary of the $4$-cube.
This produces four $3$-cubes with altogether $12$ squares where the squares
require the choice of one diagonal for triangulating them.
This cubical configuration is
invariant under $\tilde{\sigma} = (1247)(36)(0)(5)$ ($\sigma$ from Equation~\ref{eq:sigma} after antipodal identification $\pm$ modulo $15$).

\begin{figure}[p]
\includegraphics[width=.7\textwidth]{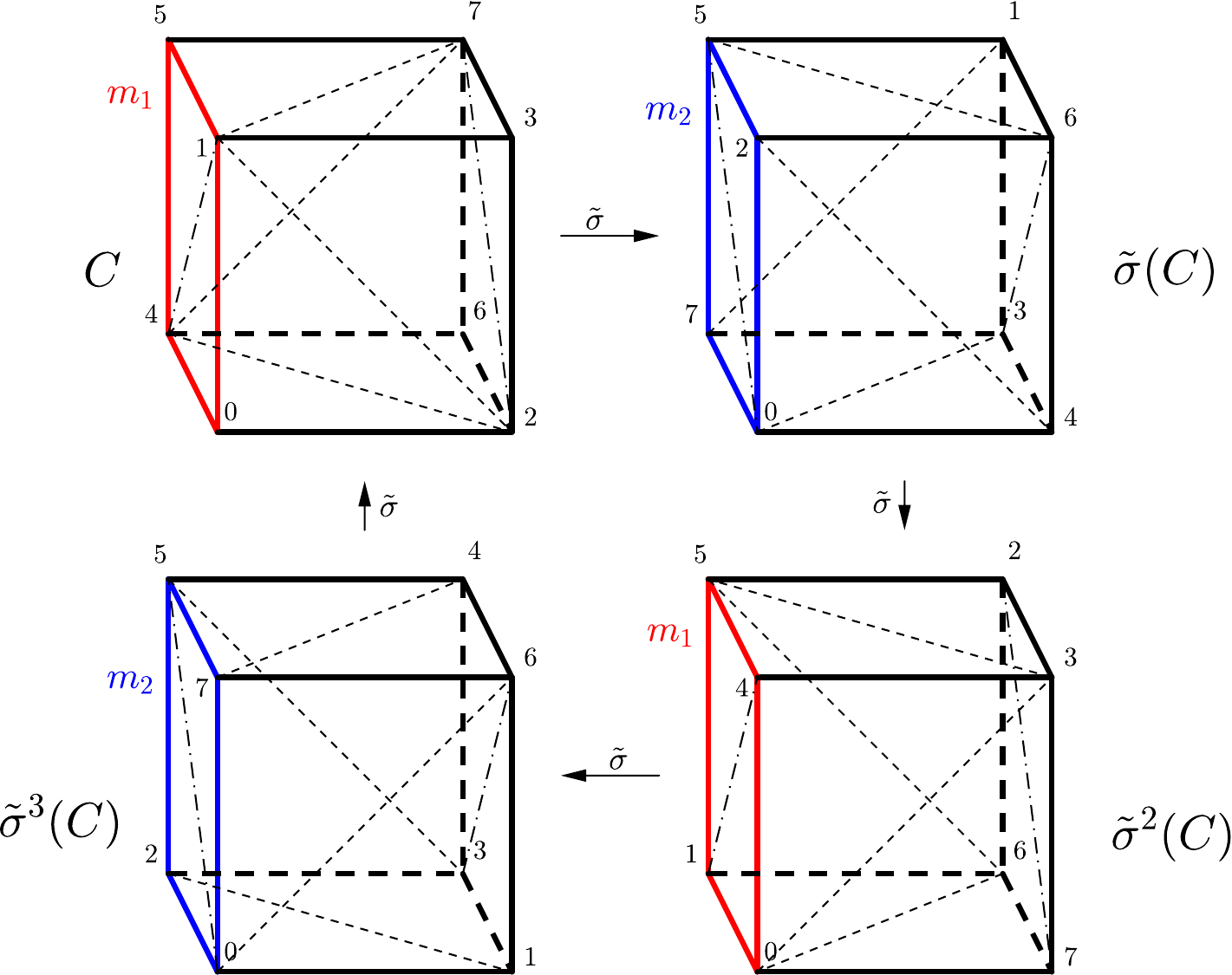}
\caption{The four $3$-cubes of the $12$-vertex equilibrium $\mathbb{R}P^3$ as an orbit of cube $C$ (top left) under $\tilde{\sigma}$. $T_1 = C \cup \tilde{\sigma}^2(C)$, as well as $T_2 = \tilde{\sigma}(C) \cup \tilde{\sigma}^{3}(C)$ are solid tori respectively. Note that $C$ can be triangulated without adding any additional vertices or edges giving rise to an $11$-vertex, and thus vertex-minimal equilibrium triangulation of $\mathbb{R}P^3$. \label{fig:sigma3cubes} \vspace{1cm}}
\end{figure}

\begin{figure}[p]
{\small \begin{tabular}{llllllll}
${\bf  ( 0\ 1\ 4\ p_0 )}$&${\bf ( 2\ 3\ 7\ p_0 )}$&\qquad\qquad&${( 0\ 1\ 2\ p_0 )}$&${( 4\ 5\ 7\ p_0 )}$&\qquad\qquad&${ ( 0\ 2\ 4\ p_0 )}$&${ ( 1\ 3\ 7\ p_0 )}$\\
${\bf  ( 1\ 4\ 5\ p_0 )}$&${\bf ( 2\ 6\ 7\ p_0 )}$&&${( 1\ 2\ 3\ p_0 )}$&${( 4\ 6\ 7\ p_0 )}$&&${ ( 2\ 4\ 6\ p_0 )}$&${ ( 1\ 5\ 7\ p_0 )}$\\
&&&&&&&\\
${( 0\ 2\ 5\ p_1 )}$&${( 1\ 3\ 6\ p_1 )}$&&$( 1\ 5\ 6\ p_1 )$&$( 0\ 3\ 4\ p_1 )$&&${ ( 0\ 2\ 4\ p_1 )}$&${ ( 1\ 3\ 7\ p_1 )}$\\
${( 0\ 5\ 7\ p_1 )}$&${( 3\ 4\ 6\ p_1 )}$&&$( 2\ 5\ 6\ p_1 )$&$( 0\ 3\ 7\ p_1 )$&&${ ( 2\ 4\ 6\ p_1 )}$&${ ( 1\ 5\ 7\ p_1 )}$\\
&&&&&&&\\
${\bf   ( 0\ 1\ 4\ p_2 )}$&${\bf  ( 2\ 3\ 7\ p_2 )}$&&$( 1\ 5\ 6\ p_2 )$&$( 0\ 3\ 4\ p_2 )$&&$\color{gray}{( 0\ 1\ 6\ p_2 )}$&$\color{gray}{( 2\ 3\ 5\ p_2 )}$\\
$ {\bf  ( 1\ 4\ 5\ p_2 )}$&${\bf  ( 2\ 6\ 7\ p_2 )}$&&$( 2\ 5\ 6\ p_2 )$&$( 0\ 3\ 7\ p_2 )$&&$\color{gray}{( 0\ 6\ 7\ p_2 )}$&$\color{gray}{( 3\ 4\ 5\ p_2 )}$\\
&&&&&&&\\
${( 0\ 2\ 5\ p_3 )}$&${( 1\ 3\ 6\ p_3 )}$&&${( 0\ 1\ 2\ p_3 )}$&${( 4\ 5\ 7\ p_3 )}$&&$\color{gray}{( 0\ 1\ 6\ p_3 )}$&$\color{gray}{( 2\ 3\ 5\ p_3 )}$\\
${( 0\ 5\ 7\ p_3 )}$&${( 3\ 4\ 6\ p_3 )}$&&${( 1\ 2\ 3\ p_3 )}$&${( 4\ 6\ 7\ p_3 )}$&&$\color{gray}{( 0\ 6\ 7\ p_3 )}$&$\color{gray}{( 3\ 4\ 5\ p_3 )}$
\end{tabular}}
\caption{The $48$ tetrahedra of the $12$-vertex nice equilibrium triangulation of $\mathbb{R}P^3$ with $12$ vertices sorted by the $3$-cube they are contained in. 
Pairs of columns contain pairs of triangulated opposite squares. \label{tab:rp3}}
\end{figure}

On the level of the ``cubical Clifford torus'', this twofold quotient
produces $8$ squares, as depicted in Figure~\ref{fig:grid} on the right.
They occur in the equilibrium $\mathbb{R}P^3$ as the union
of two orbits of squares under $\tilde{\sigma}$. It is not polyhedral, but
it is the boundary of the union of two of the four $3$-cubes.
The remaining $4$ squares represent the
standard Heegaard diagram belonging to the associated
decomposition of $\mathbb{R}P^3$ into two solid tori,
and Figure~\ref{fig:sigma3cubes} shows
its triangulated version.
This leads to the $f$-vector
\[(8 + 4, 28 + 4 \cdot 8, 4 \cdot 18 +2 \cdot 12,4 \cdot 12) = (12, 60, 96, 48).\]
A list of the $48$ tetrahedra is given in Table~\ref{tab:rp3}.

Unfortunately, a choice of diagonals for the $12$ squares is impossible if we require
an invariance under $\tilde{\sigma}$ or even just under $\tilde{\sigma}^2$.
This is the same problem as discussed above 
for the complex projective $3$-space and its real part:
The diagonal $(0\ 5)$ in the square $(0\ 1\ 5\ 4)$ cannot be replaced by $(1\ 4)$
since then $(2\ 4)$ occurs in two squares $(2\ 5\ 4\ 3)$ and $(0\ 2\ 6\ 4)$ where the diagonal
$(0\ 6)$ in the latter is also impossible if $(0\ 3)$ is a diagonal in $(0\ 1\ 3\ 2)$, see Figure~\ref{fig:grid}.

This construction also appears in \cite{BasakSarkar}, see Figure 7 therein for a picture.
\end{proof}

\begin{rem}
  Note that the construction from Proposition~\ref{prop:equilibriumRP3} is not optimal in the sense that we can save an additional vertex. This can be deduced from looking at cube $C$ depicted in the top-right of Figure~\ref{fig:sigma3cubes}. This cube can be triangulated without adding any vertex or edge. Removing the vertex in the centre of $C$ leads to a vertex-minimal $11$-vertex equilibrium triangulation of $\mathbb{R}P^3$ with $f$-vector $(11, 52, 82, 41)$.
\end{rem}

\medskip

\begin{theorem}[From the 24-cell to an equilibrium
    triangulation of $\mathbb{R}P^3$]
  \label{thm:24cellequilibrium}
  The $12$-vertex equilibrium triangulation of $\mathbb{R}P^3$
  from Proposition~\ref{prop:equilibriumRP3} is a subdivision
  of the $24$-cell under antipodal identification. 
\end{theorem}

\begin{proof}
Two pyramids over the same square can be regarded as two halfs
of an octahedron. This is still true if the square is triangulated
by a diagonal. In this way the 48 tetrahedra of an equilibrium $\mathbb{R}P^3$
can be regarded as 24 pyramids and as 12 (abstract) octahedra. By the regularity
of this (abstract) complex its universal covering must coincide
with the 24-cell. In particular, this equilibrium triangulation
can be obtained from the boundary complex of a convex 4-polytope by an appropriate subdivision.
\end{proof}

\medskip

\noindent
{\bf A nice equilibrium triangulation of $\mathbb{R}P^4$.}

For $k=4$ the situation is more complicated: We must triangulate
$20$ cubes with $40$ squares in a coherent way. 
In short, our approach is to choose a triangulation of each $3$-cube and fix incoherently
triangulated squares of adjacent $3$-cubes by introducing an additional
flat tetrahedra acting as an ``adaptor''. This leads to
the following theorem.

\begin{thm}
  \label{thm:nicerp4}
  There is a nice equilibrium triangulation of $\mathbb{R}P^4$ (with $21$ vertices).
\end{thm}

\begin{proof}
We start with the boundary complex of a $5$-cube with $f$-vector
$(32, 80, 80, 40, 10)$. After identification of antipodal points we obtain
a cubical decomposition of $\mathbb{R}P^4$ with the $f$-vector
$(16, 40, 40, 20, 5)$. Any two of the five $4$-cubes intersect in
a pair of opposite $3$-cubes.

\medskip
The idea of the construction is to triangulate
the $3$-skeleton of this complex and then to add one extra vertex in each of the
five $4$-cubes. According to \cite{PVR} there are several triangulations
of a $3$-cube. We choose the one with five tetrahedra and without
a main diagonal, which exists in two mirror symmetric
versions.

It follows that the $40$ squares inherit choices of their
diagonals. However, these are not necessarily coherent.
Any square is a face of three distinct $3$-cubes. If for a square
all three induced diagonals coincide, no further action is necessary.
Otherwise, two of the three diagonals coincide and a third one does not.
In this case we must insert a ``flat'' tetrahedron with the same four
vertices as this square which adapts the two contradicting diagonals
to one another. Combinatorially, these tetrahedra are not flat but ordinary.
However, they can never arise from an embedding of the cubical
complex into any euclidean space.

Choosing the first of the two mirror symmetric versions for all $20$ $3$-cubes
results in exactly $20$ incoherently triangulated squares.
Iterating over all $2^{20}$ possible choices of mirror symmetric triangulations of the 
cubes reveals that this is the smallest possible number.

Introducing $20$ additional ``flat'' tetrahedra to fix these inconsistencies, 
the $f$-vector of the triangulated 3-dimensional cubical complex is $(16,100,200,120)$,
and the final $f$-vector of the triangulated $\mathbb{R}P^4$, obtained by coning over the triangulated boundaries of the $4$-cubes, is
$(21, 180, 520, 600, 240)$.

\medskip
A list of facets for this triangulation is given in Appendix~\ref{app:rp4}.
\end{proof}

\medskip
The 5-cube $C^5$ can be regarded as a Cartesian product of a square and a cube:
$C^5 = C^2 \times C^3$. Therefore its boundary decomposes into the two
parts $\partial C^2 \times C^3$ and $C^2 \times \partial C^3$, separated
by the common hypersurface
$\partial C^2 \times \partial C^3 \cong S^1 \times S^2$.
After identification of opposite points we obtain a hypersurface
of type $(S^1 \times S^2)/\pm $ which is the total space of the
twisted 2-sphere bundle over $S^1$,
also called {\it $3$-dimensional Klein bottle}. In differential geometry
it arises as the tensor product $S^1 \otimes S^2 \subset \mathbb{R}^6$,
a tight and taut submanifold \cite[Cor.5.6]{Ku1}, in combinatorial topology
it occurs as the unique non-spherical 3-manifold that admits a
(simplicial) triangulation with 9 vertices \cite{Lu}. This unique
triangulation can
be represented as the permcycle {\bf 1~1~2~5} \cite {Ku-La-permcycle}.

\medskip
In the cubical $\mathbb{R}P^4$ this hypersurface
arises as the boundary
of the union of any two 4-cubes among the five 4-cubes. They intersect
in two opposite 3-cubes, therefore the 3-dimensional Klein bottle
is decomposed into 12 cubes.
These split into $6 + 6$ cubes with an ordinary Klein bottle in between,
depicted in Figure~\ref{fig:KleinBottle} in the triangulated version.

\begin{figure}
\includegraphics[width=.4\textwidth]{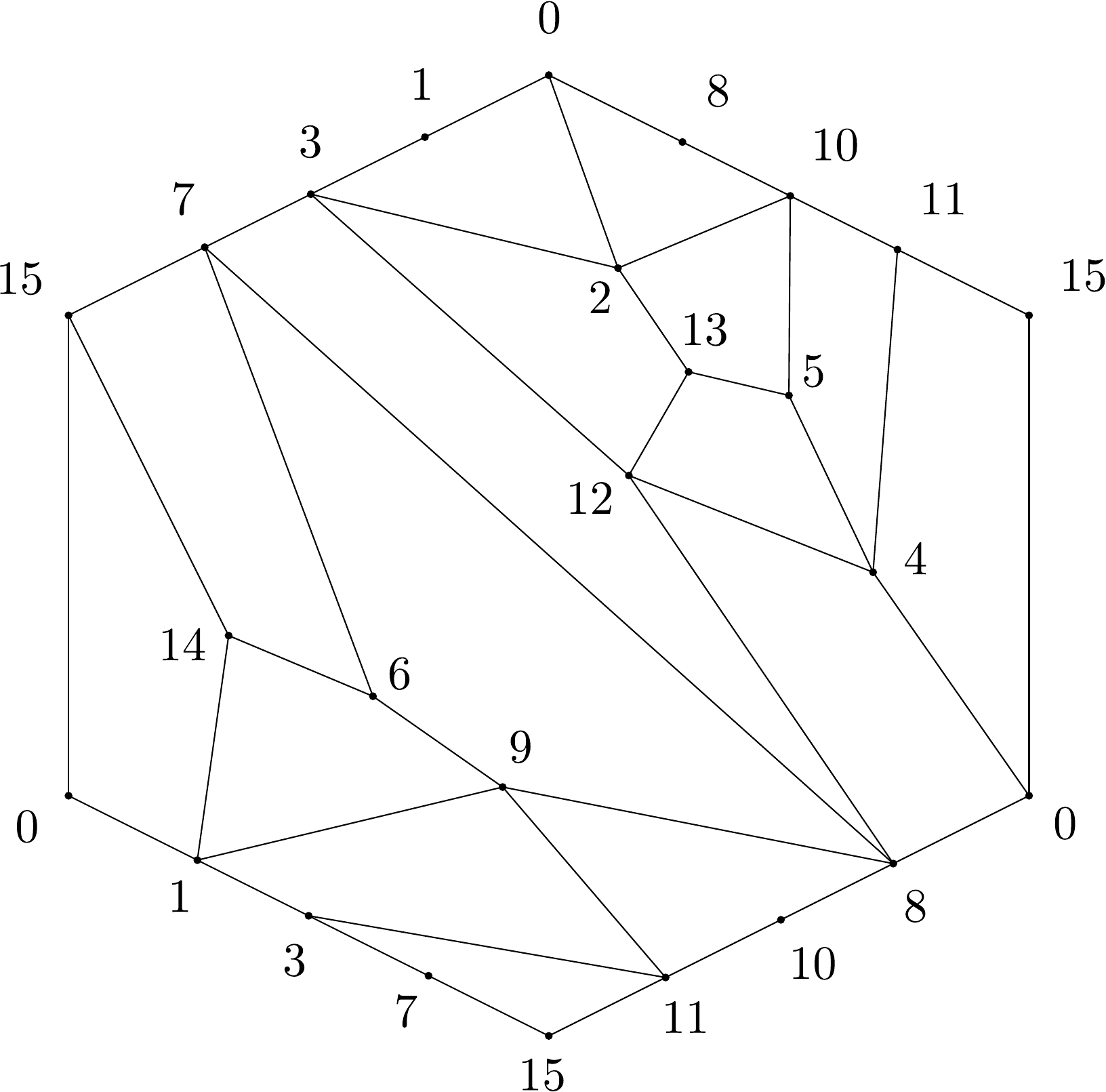}
\caption{Klein bottle bounding the union of six (triangulated) cubes in $\mathbb{R}P^4$. The complex has $16$ vertices and $16$ quadrilaterals. \label{fig:KleinBottle}}
\end{figure}

\begin{cor}
  The equilibrium triangulation of $\mathbb{R}P^{4}$ contains a $3$-dimensional Klein bottle
  as a $3$-dimensional subcomplex together
  with the total space of a disc bundle on either side over
  $\mathbb{R}P^{1}$ and $\mathbb{R}P^{2}$, respectively.
\end{cor}

\medskip

\noindent
    {\bf The case $k=5$:}
    We have the polyhedral decomposition into 60 pyramids over
    a 4-cube with altogether $2^5 + 6 = 38$ vertices. A coherent
    triangulation of all these $4$-cubes without additional vertices
    seems to be a difficult task, and adapting two
    distinct triangulations of a $3$-cube to one another is much more difficult than in the $4$-dimensional case.
However, one can still adapt triangulations of squares, and after
triangulating all 3-cubes one can add a vertex at the centre of each 4-cube --
in addition to the ones at the centre of each 5-cube.
This would lead to an equilibrium triangulation of $\mathbb{R}P^5$ with
$2^5 + 6 + 30 = 68$ vertices. If some of the $4$-cubes
can be coherently triangulated, this number reduces.

\medskip

\noindent
{\bf The general case.}
In the context of isoparametric hypersurfaces in spheres, the Cartesian products
$S^k \times S^l$ occur as hypersurfaces in the $(k+l+1)$-sphere
sitting half way between the two focal sets $S^k$ and $S^l$.
After identification of antipodal points we obtain a similar 
decomposition for $\mathbb{R}P^{k+l+1}$. In this case the
hypersurface is $S^k \times S^l / \pm$ as the total space of
an $S^l$-bundle over $\mathbb{R}P^{k}$ on one side, and as an $S^k$-bundle
over $\mathbb{R}P^{l}$ on the other.
Our cubical decomposition provides cubical analogues, and their triangulations
provide simplicial analogues.

\begin{que}
  Is there a nice equilibrium triangulation of $\mathbb{R}P^k$
for any $k$? This was claimed in \cite[Theorem 3.24]{BasakSarkar},
but the idea of the proof explained there requires that
the $(k-1)$-dimensional complex
to be triangulated is embeddable into some euclidean space. This does
not seem to be true. 

Without such an embedding
the main difficulty is to find a coherent triangulation of all
these $(k-1)$-cubes. Symmetry induced by the
automorphism $\sigma$ cannot help here, and
an induction from $k$ to $k+1$ is not obvious. 
Moreover, in higher dimensions there are very large numbers of
distinct triangulations of cubes providing an enormous number
of potential solutions to investigate.
\end{que}

\section{From the 24-cell to a tight polyhedral embedding
of $\mathbb{R}P^3$}
\label{sec:tight}

An embedding  $f \colon M \rightarrow \mathbb{E}^d$ of a compact manifold 
$M$ into euclidean space is called {\it tight}, if for any open or closed 
half-space $E_+ \subset \mathbb{E}^d$ the induced homomorphism
\[H_*(f^{-1}(E_+);\mathbb{F}) \longrightarrow H_*(M;\mathbb{F})\]
is injective for some field $\mathbb{F}$,
where $H_*$ denotes the ordinary singular homology \cite{Ku1}. 
For $\mathbb{R}P^k$, due to its $2$-torsion, an embedding can only be
tight if $\mathbb{F} = \mathbb{Z}_2$ is the field with $2$ elements,
for $\mathbb{C}P^k$ we can choose $\mathbb{F}$ freely (or even work with 
integral homology). 

For smooth and polyhedral submanifolds
an equivalent formulation of a tight embedding is that any ``height function''
(think of regular simplexwise linear functions in the sense of \cite{Ku1}) 
has the minimum number of critical points, counted with multiplicity,
and that number coincides with the sum of all Betti numbers \cite[3.15]{Ku1}.

\medskip

When talking about tight embeddings of polyhedral structures -- such as triangulations --
it makes sense to consider (polyhedral) embeddings into the boundary complex of 
an ambient polytope. Every intersection of the triangulation with a half-space
has the homotopy type of a subcomplex of the triangulation induced by some 
subset of its vertices. At the same time, the coordinates of the 
vertices of the polytope determine, which subsets of vertices correspond 
to a half-space. This provides us with a finite list of half-spaces to 
check in order to determine tightness of an embedding. 

This framework can then also be used to determine that {\em all} possible 
polyhedral embeddings of a given triangulation into a given polytope must be
tight, as is the case below. Naturally, this condition is more restrictive if more
vertex subsets correspond to half-spaces. It is thus strongest when the ambient polytope
is the simplex of dimension one less the number of vertices of the triangulation.
Accordingly, an $n$-vertex triangulation is said to be {\em tight}, if tightness for an embedding 
into the $(n-1)$-simplex can be shown.

Recently, a condition of a triangulated manifold equivalent to tightness
was found in purely algebraic terms of the Stanley-Reisner ring
\cite{IK}.

\medskip

For $\mathbb{R}P^2$ and $\mathbb{C}P^2$ there are tight embeddings
based on tight triangulations of these manifolds \cite[1.4]{Ku1}:
By definition, given an $n$-vertex tight triangulation, its maximum codimension embedding 
into the skeleton of an $(n-1)$-dimensional simplex is tight.
 No tight triangulation of $\mathbb{R}P^3$ exists \cite[Thm.5.3]{Ku1}.
The smallest equilibrium triangulations of $\mathbb{R}P^2$ ($7$ vertices) 
and $\mathbb{C}P^2$ ($10$ vertices) are not tight triangulations.
Nonetheless they also lead to tight embeddings into spaces of smaller codimension,
namely $4$-space ($\mathbb{R}P^2$) and $7$-space ($\mathbb{C}P^2$) \cite{BK}.
To date, no tight polyhedral embeddings of $\mathbb{R}P^3$ or $\mathbb{C}P^3$ 
are known\footnote{The remark in a footnote in \cite[p.120]{Kui}
was based on an error. This will now be repaired by
Theorem~\ref{tightRP3}.}. 

In this section we present a tight embedding of a triangulation of
$\mathbb{R}P^3$ into $6$-space.
Unfortunately, the equilibrium triangulations of $\mathbb{R}P^3$ from 
Section~\ref{sec:equilreal} are not suitable for a tight embedding.
Instead, we work with a very symmetric $12$-vertex triangulation of 
$\mathbb{R}P^3$. Incidentally, it arises from a centrally-symmetric 
simplicial subdivision of the 
$24$-cell by adding diagonals of octahedra, just like the 
equilibrium triangulation of $\mathbb{R}P^3$ from Proposition~\ref{prop:equilibriumRP3},
cf. Theorem~\ref{thm:24cellequilibrium}.

\begin{lem}[F.H.Lutz \cite{Lu}]
\label{RP3}

There exists a triangulation of $\mathbb{R}P^3$ with $12$ vertices,
denoted by $P_{12}$, admitting a vertex transitive automorphism group
of order $24$. The group is generated by the permutations
\begin{eqnarray*}
T &=& (1 \ 3 \ 5 \ 7 \ 9 \ 11)(2 \ 4 \ 6 \ 8 \ 10 \ 12)\\
R &=& (1 \ 10)(2 \ 5)(3 \ 12)(4 \ 7)(6 \ 9)(8 \ 11)\\
S &=& (1 \ 11)(2 \ 10)(3 \ 9)(4 \ 8)(5 \ 7)(6)(12),
\end{eqnarray*}
and $P_{12}$ is generated by the three tetrahedra
$(1 \ 2 \ 3 \ 5), \ (1 \ 2 \ 5 \ 10), \ (1 \ 2 \ 10 \ 11)$ under action of this group. 
The first one generates an orbit
of $24$ tetrahedra, the others orbits contain $12$ tetrahedra each.
Altogether, the $f$-vector of $P_{12}$ is $(12, 60, 96, 48)$. 
The six missing edges are 
\[(1 \ 7), \ (2 \ 8), \ (3 \ 9), \ (4 \ 10), \ (5 \ 11), \ (6 \ 12),\] 
thus the permutation $T^3$ can be interpreted as a fixed-point free involution.
\end{lem}

We remark here that the edges $(1 \ 2)$, $(1 \ 8)$ 
and -- consequently -- all edges in their $T$-orbits
are $4$-valent. The star of such an edge is nothing but
an octahedron, subdivided by one diagonal. These $12$ edge stars cover the
complex entirely. Moreover, these edges form three cycles of length four:
$(1 \ 2 \ 7 \ 8), (3 \ 4 \ 9 \ 10), (5 \ 6 \ 11 \ 12)$. 
This has the consequence that omitting these $12$ edges in $P_{12}$
leads to a polyhedral decomposition of $\mathbb{R}P^3$ into $12$ octahedra,
with the link of every vertex transforming into the boundary of a $3$-cube.
Therefore, the universal covering of this decomposition has the Schl\"afli symbol
$\{3,4,3\}$ and thus must coincide with the boundary complex of the 
$24$-cell with $24$ vertices and $24$ octahedra.
However, only the subgroup generated by $T$ is a subgroup of both
the automorphism group of $P_{12}$ and the one of the 
$12$ octahedra.

\medskip

\begin{theorem}[Tight polyhedral embedding of $\mathbb{R}P^3$]
 \label{tightRP3} 
 
The triangulation $P_{12}$ of $\mathbb{R}P^3$ from  Lemma~\ref{RP3} admits
a tight polyhedral embedding into the $3$-skeleton of the boundary complex
of the $6$-dimensional cross polytope $\beta^6$ (= the dual of the $6$-cube),
where the six missing edges of $P_{12}$ correspond to the six diagonals of $\beta^6$. 
\end{theorem}

\begin{proof}
Up to symmetries of the cross polytope, the embedding of the $12$-vertex 
triangulation of $\mathbb{R}P^3$ from Lemma~\ref{RP3} into $\beta^6$
is unique. For the tightness it is sufficient to show that the span 
of any subset of vertices containing none of the antipodal pairs
constituting the diagonals injects into the manifold at the 
$\mathbb{Z}_2$-homology level. By duality, $H_0$ and $H_1$
 (for $0$- and $1$-tightness tightness respectively) are sufficient
 to show the tightness of the embedding.
  
The $f$-vector of the boundary complex of $\beta^6$ is $(12, 60, 
160, 240, 192, 64)$. The $0$-tightness holds because $P_{12}$ contains 
the entire edge graph of $\beta^6$. That is, for every subset of vertices $W$,
the subcomplex of $P_{12}$ spanned by $W$ has exactly as many connected 
components as the subcomplex of $\beta^6$ spanned by $W$, and $0$-homology injects. 

To show $1$-tightness we must prove that $1$-homology injects for 
subcomplexes spanned by subsets of up to six vertices -- the remaining vertex 
sets follow by duality (see below).
This is trivially the case for subsets of size at most two. For larger subsets we prove:

\begin{enumerate}
  \item Any ``empty" triangle in the closure of $\beta^6 \setminus P_{12}$ represents
the generator of the first $\mathbb{Z}_2$-homology (or fundamental group)
of $\mathbb{R}P^3$. 
  \item Any tetrahedron not in $P_{12}$ spans the union of
two triangles of $P_{12}$ representing homology between the two other triangles
which are not in $P_{12}$. Note that this, together with the previous step, and the fact that $P_{12}$ covers the entire graph of $\beta^6$ implies that $1$-homology injects.  
  \item Any vertex set of a $4$- or $5$-simplex of $\beta^6$ spans the generator in $1$-homology in $P_{12}$ and nothing else.
\end{enumerate}

These statements are tedious to verify by hand, but straightforward to 
check by a computer. Here, we give additional information to facilitate this
computer check.

\medskip

The $64$ triangles in $\beta^6$ that are ``empty'' triangles in $P_{12}$ split
into four orbits under the automorphism group. They are generated by
\[(1 \ 2 \ 4), \ (2 \ 3  \ 4), \ ( 1 \ 2 \ 6), \ (2 \ 6 \ 10)\]
of size $24$, $12$, $24$, and $4$ respectively. To show that one of them
is non-trivial, construct a discrete Morse function
collapsing $P_{12}$ onto one of these empty triangles with only one critical tetrahedron
and triangle in the process. Such a function is quite simple to find.

If one of the empty triangles is non-trivial in homology, then all other empty 
triangles in its orbit are non-trivial. Moreover, note that
the span of $\{ 1,2,3,4\}$ in $P_{12}$ is $(1\ 2\ 3),\ (1\ 3\ 4),\ (2\ 4)$
and hence $(1\ 2\ 4)$ and $(2\ 3\ 4)$ are homologous. Analogously, we can
verify that $(1\ 2\ 4)$ and $(1\ 2\ 6)$, and $(1\ 2\ 6)$ and $(2\ 4\ 6)$
are homologous. Altogether, it follows that all $64$ empty triangles represent
the non-trivial element in the first $\mathbb{Z}_2$-homology of $P_{12}$. 

\medskip

The $192$ tetrahedra of $\beta^6$ that are not tetrahedra of $P_{12}$ split
into the following generators of orbits under the automorphism group:
\[ (1\ 2\ 3\ 4),\ (1\ 2\ 3\ 6),\ (1\ 2\ 4\ 6),\ (1\ 2\ 4\ 12),\ (1\ 2\ 5\ 6),\ (1\ 2\ 5\ 9),\ (1\ 2\ 5\ 12) \]
of length $24$, and 
\[ (1\ 2\ 4\ 5),\ (1\ 2\ 4\ 11)\]
of length $12$.
Each of these $4$-tuples spans two triangles and two empty triangles of $P_{12}$
and - therefore - represents the generator of the first 
$\mathbb{Z}_2$-homology.

\medskip
The $192$ $4$-simplices of $\beta^6$ are generated by the following $5$-tuples:
\[ (1\ 2\ 3\ 4\ 6),\ (1\ 2\ 3\ 4\ 11),\ (1\ 2\ 3\ 5\ 6),\ (1\ 2\ 3\ 5\ 10),\ (1\ 2\ 4\ 5\ 6),\ (1\ 2\ 5\ 6\ 9) \]
of length $24$, and
\[ (1\ 2\ 3\ 4\ 5),\ (1\ 2\ 3\ 4\ 12),\ (1\ 2\ 3\ 6\ 10),\ (1\ 2\ 6\ 10\ 11) \]
of length $12$.
\begin{itemize}
  \item $(1\ 2\ 3\ 4\ 5)$, $(1\ 2\ 3\ 5\ 6)$, and $(1\ 2\ 3\ 5\ 10)$ span two adjacent tetrahedra and one extra edge in $P_{12}$,
a complex that collapses onto an empty triangle.  
  \item $(1\ 2\ 3\ 4\ 6)$, $(1\ 2\ 3\ 6\ 10)$, and $(1\ 2\ 4\ 5\ 6)$
span a triangulated 5-vertex M\"obius band which collapses onto an empty triangle.
  \item $(1\ 2\ 3\ 4\ 11)$, $(1\ 2\ 3\ 4\ 12)$, $(1\ 2\ 5\ 6\ 9)$, and $(1\ 2\ 6\ 10\ 11)$ contain one tetrahedron 
and the cone from the $5^{th}$ vertex to two opposite edges of the tetrahedron, hence 
they, again, represent the generator of the first homology. 
\end{itemize}

There are four orbits of the $64$ facets (5-simplices) of $\beta^6$
under the group with the following generators:
\[(1 \ 2 \ 3 \ 4  \ 5 \ 6), \ (1 \ 2 \ 3 \ 4  \ 6 \ 11), \ 
(1 \ 2 \ 3 \ 6  \ 10 \ 11), \ (1 \ 2 \ 5 \ 6  \ 9 \ 10)\]  
Each orbit also contains the complementary $6$-tuple by the involution $T^3$.
The first three span three stacked tetrahedra and two additional triangles,
and collapse onto one of its eight empty triangles. The last one
spans three tetrahedra pairwise sharing an edge. It, again collapses
onto on of its eight empty triangles.

\medskip
We have now verified that $1$-homology injects for all half-spaces spanned by 
up to six vertices of $\beta^6$. 
By Alexander duality it is not necessary to discuss $7$-tuples, $8$-tuples 
or $9$-tuples. Their homological type is complementary to the one of
their complements. 
As an example, the complex opposite to the triangle $(1\ 9\ 10)$ 
collapses onto a projective plane:
It contains the M\"obius band with the following $8$ triangles
\[(5\ 8\ 12),\ ( 2\ 5\ 12),\ (2\ 5\ 6),\ (2\ 4\ 6),\ (4\ 6\ 11),\ (3\ 6\ 11),\ (3\ 8\ 11),\ (3\ 8\ 12)\]
and the cone from $7$ to its boundary $(2\ 4\ 11\ 8\ 5\ 6\ 3\ 12)$.
\end{proof}

\section{The highly symmetric and vertex-minimal triangulation of
  $\mathbb{R}P^4$}
  \label{sec:min}
  
Recall that there is no $15$-vertex triangulation of $\mathbb{R}P^4$ \cite{AM}.
There is a standard triangulation of $\mathbb{R}P^k$ with $2^{k+1} - 1$
vertices as a 2-fold quotient of the first barycentric subdivision of
a $(k+1)$-simplex \cite{AAK}. Combinatorially this appears as the
vertex link in the standard lattice triangulation of $\mathbb{R}^{k+1}$
(or standard triangulation of the cubical tessellation) \cite{Br-Ku2,Mara}.
More recently, triangulations of $\mathbb{R}P^k$ with smaller vertex numbers
have been constructed, see \cite{VZ} and \cite{AAK}. The 
latter article even presents triangulations with vertex numbers subexponential 
in $k$. The best asymptotic bound for large $k$ still seems to be unknown. 

However, for $k=4$ there is a surprising and exceptional example. The example as a
combinatorial triangulation with vertex transitive automorphism group is due to Lutz \cite{Lu}.
Here we present it as the quotient of a centrally symmetric $5$-polytope under antipodal 
identification, thereby giving a separate, independent proof of its existence
as an abstract triangulation.

\medskip

\begin{theorem}[Vertex-minimal triangulation of $\mathbb{R}P^4$ \cite{Lu} as quotient of centrally symmetric $5$-polytope]
\label{RP4}

There exists a triangulation of $\mathbb{R}P^4$
with $16$ vertices admitting the automorphism group $S_6$
of order $720$ with two orbits of vertices
$\{1, 2, 3, \ldots, 10\}$ and  $\{{\it 11, 12, 13, 14, 15, 16}\}$.
The group is generated by the permutations
\begin{eqnarray*}
T &=& (1 \ 2 \ 3 \ 4 \ 5 \ 10)(6 \ 8 \ 9)(7)({\it 11 \ 12 \ 13  \ 14  \ 15  \ 16})\\
S &=& (2 \ 7)(4 \ 10)(5 \ 6)({\it 11 \ 12})
\end{eqnarray*}
and the triangulation is determined by the two $4$-simplices
$(1 \ 2 \ 4 \ 5 \ {\it 11}), \ (1 \ 2 \ 4 \ {\it 11 \ 13})$ generating
orbits of length $30$ and $120$, respectively. The triangulation is
$2$-neighbourly, and the $f$-vector is $(16, 120, 330, 375, 150)$.
This triangulation appears in \cite{Lu}.

The universal covering of this triangulation is the
boundary complex of a spherical and centrally symmetric convex $5$-polytope $P$
with $32$ vertices
that realises the full automorphism group $S_6 \times C_2$
by euclidean symmetries.
\end{theorem}

\begin{rem}
A description of the abstract 4-sphere as a double covering
of the 16-vertex $\mathbb{R}P^4$ together with combinatorial observations
on the Witt design and other constructions of the 16-vertex triangulation
can be found in \cite{Bal}. There the construction of the 4-sphere is based 
on bistellar flips but the polytopality is not mentioned. 
\end{rem}

\begin{proof}
The (abstract) triangulation itself is described in \cite{Lu}. 
Here we describe it as a quotient of a centrally symmetric $5$-polytope
embedded into a hyperplane in Euclidean $6$-space.

First, we adapt the labelling from \cite{Lu} to better suit our purposes.
Namely,  we relabel the vertices in
the small vertex orbit ${\it 11, 12, 13, 14, 15, 16}$ by $1,2,3,4,5, 6$, and 
the vertices in the large orbit by the set of all $10$ pairs of disjoint triples out
of these six vertices. This implies

\begin{center}
\begin{tabular}{lllll}
$1 \mapsto (124)(356)$, & $2 \mapsto (146)(235)$, & $3 \mapsto (125)(346)$, &$4 \mapsto (145)(236)$, &$5 \mapsto (134)(256)$,  \\
$6 \mapsto (156)(234)$, & $7 \mapsto (135)(246)$, & $8 \mapsto (126)(345)$, &$9 \mapsto (123)(456)$, &$10 \mapsto (136)(245)$
\end{tabular}
\end{center}

The $5$-polytope $P$ has the vertices $1, 2, 3, 4, 5, 6, 1', 2', 3', 4', 5', 6'$
and triples $(abc)(def)$ splitting into vertices $(abc)(d'e'f')$ and $(a'b'c')(def)$.
This leads to an action of $S_6\times C_2$ generated by
\begin{eqnarray*}
T &=& (1 \ 2 \ 3 \ 4 \ 5 \ 6)(1' \ 2' \ 3' \ 4' \ 5' \ 6')\\
S &=& (1 \ 2)(1' \ 2')\\
R &=& (1 \ 1')(2 \ 2')(3 \ 3')(4  \ 4')(5 \ 5')(6 \ 6')
\end{eqnarray*}
and the natural extension to the triples:
\[T(a \ b \ c) := (Ta \ Tb \ Tc), \ \ T(a' \ b' \ c') := (Ta' \ Tb' \ Tc')\] 
\[S(a \ b \ c) := (Sa \ Sb \ Sc), \ \ \ S(a' \ b' \ c') := (Sa' \ Sb' \ Sc')\]
\[R(a \ b \ c) := (Ra \ Rb \ Rc), \ \ R(a' \ b' \ c') := (Ra' \ Rb' \ Rc')\]
After dividing out by the central involution $R$ this action
is equivalent to the one on the $16$ vertices in the original triangulation of $\mathbb{R}P^4$.

\medskip
The idea of the construction is to have vertices $1, 2, 3, 4, 5, 6$ forming
a regular simplex in 5-space, and the other ten vertices
corresponding to the barycenters of triples (explained in detail below). 
The vertices $1', 2', 3', 4', 5', 6'$
form the simplex opposite the regular $5$-simplex with its corresponding triples.
It is convenient to work in the hyperplane of $\mathbb{R}^6$ 
where the sum of all coordinates is zero.
Then a regular $5$-simplex is defined by the six vertices
\begin{eqnarray*}
1 &\leftrightarrow& \textstyle\frac{1}{\sqrt{5}}(5, -1, -1, -1, -1, -1)\\
2 &\leftrightarrow & \textstyle\frac{1}{\sqrt{5}}(-1, 5, -1, -1, -1, -1)\\
3 &\leftrightarrow& \textstyle\frac{1}{\sqrt{5}}(-1, -1, 5, -1, -1, -1)\\
4 &\leftrightarrow & \textstyle\frac{1}{\sqrt{5}}(-1, -1, -1, 5, -1, -1)\\
5 &\leftrightarrow & \textstyle\frac{1}{\sqrt{5}}(-1, -1, -1, -1, 5, -1)\\
6 &\leftrightarrow & \textstyle\frac{1}{\sqrt{5}}(-1, -1, -1, -1, -1, 5)
\end{eqnarray*}
admitting an obvious $S_6$-action. The vertices $1', 2', 3', 4', 5', 6'$
are defined as antipodes, so $1'$ has the coordinates 
$\textstyle\frac{1}{\sqrt{5}}(-5, 1, 1, 1, 1, 1)$, and so on.
A triple $(a\ b\ c)$ corresponds to the barycenters of
the three points $a$, $b$, and $c$ but in such a way that all points lie in 
a sphere of radius $\sqrt{6}$. Note that, this way, for instance, $123$ and 
$4'5'6'$ are both assigned coordinates
$(1, 1, 1, -1, -1, -1)$, and $1'2'3'$ and $456$ are both assigned coordinates
$(-1, -1, -1, 1, 1, 1)$ (hence the grouping together of pairs of complementary 
triples in this way). Altogether we have $6 + 6 + 20 = 32$ points
in the hyperplane, and the polytope $P$ (inscribed into the sphere of 
radius $\sqrt{6}$) is defined as their convex hull.
By construction, $P$ is invariant under the $S_6$-action and under
the antipodal involution sending any point to its negative.
It remains to determine the facets of $P$.

\medskip

The $5$-space $\mathbb{R}^5$ containing $P$ is the graph
of a linear function $\mathbb{R}^5 \to \mathbb{R}$
assigning the sixth coordinate as the negative sum of the others.
Therefore we can just forget the sixth coordinate and obtain
a polytope $P'$ in $\mathbb{R}^5$ which is combinatorially
isomorphic with $P$ by an affine transformation.
We compute the facets of this $P'$.

\medskip
We now show that the $5$-tuples $\Delta_1 = (1, \ (1 2 3)(4'5'6'), \ (1 2 4)(3'5'6'), \ (1 2 5)(3'4'6'), \ (1 2 6)(3'4'5'))$ and 
$\Delta_2 = (1, \ 2', \ (1 3 4)(2'5'6'), \ (1 3 5)(2'4'6'), \ (1 4 5)(2'3'6'))$ 
are facets of $P'$, generating orbits of $60+240=300$ facets under the automorphism group.
 
The normal vector $N_1 = (1 + \frac{1}{\sqrt{5}}, \sqrt{5}-1, 0, 0, 0)$
is orthogonal to $\Delta_1$:

\[N_1 \cdot {1} = \sqrt{5}\Big(1 + \frac{1}{\sqrt{5}}\Big) - \frac{1}{\sqrt{5}}\Big(\sqrt{5} - 1\Big) = \sqrt{5} + 1 - 1 + \frac{1}{\sqrt{5}} = \frac{6}{\sqrt{5}} = 2,68328 \ldots \]

\[N_1 \cdot {1 2 3} = (1 + \frac{1}{\sqrt{5}}) +(\sqrt{5} - 1) = \sqrt{5} + \frac{1}{\sqrt{5}} = \frac{6}{\sqrt{5}} = N_1 \cdot {1 2 4} =N_1 \cdot {1 2 5} =N_1 \cdot {1 2 6}.\]

Any other vertex $X$ of $P$ gives a strictly smaller value for $N_1 \cdot X$.
%

The normal vector 
 $N_2 = (\frac{3}{4} + \frac{\sqrt{5}}{6 - 2\sqrt{5}}, 
\frac{3}{4} - \frac{\sqrt{5}}{6 - 2\sqrt{5}}   , 1, 1, 1)$ 
is orthogonal to $\Delta_2$:

$N_2 \cdot {1} = \sqrt{5}\Big(\frac{3}{4} + \frac{\sqrt{5}}{6 - 2\sqrt{5}}\Big) - \frac{1}{\sqrt{5}}\Big(\frac{3}{4} - \frac{\sqrt{5}}{6 - 2\sqrt{5}}\Big)
 - \frac{3}{\sqrt{5}} = \frac{6}{6 - 2\sqrt{5}} + \frac{1}{\sqrt{5}}\Big(\frac{15}{4} - \frac{3}{4} -3 \Big) = \frac{3}{3-\sqrt{5}} = 3,92705 \ldots$

$N_2 \cdot 2'  = 
\frac{1}{\sqrt{5}}\Big(\frac{3}{4} + \frac{\sqrt{5}}{6 - 2\sqrt{5}}\Big) - \sqrt{5}\Big(\frac{3}{4} -  \frac{\sqrt{5}}{6 - 2\sqrt{5}}\Big)
 + \frac{3}{\sqrt{5}} = \frac{6}{6 - 2\sqrt{5}} + \frac{1}{\sqrt{5}}\Big(\frac{3}{4}-\frac{15}{4}  + 3 \Big) = \frac{3}{3-\sqrt{5}} $

$N_2 \cdot {1 3 4} = \frac{3}{4} + \frac{\sqrt{5}}{6 - 2\sqrt{5}} - \Big(\frac{3}{4} - \frac{\sqrt{5}}{6 - 2\sqrt{5}}\Big) + 1 + 1 - 1 = 1 + \frac{\sqrt{5}}{3-\sqrt{5}} = \frac{3}{3-\sqrt{5}} = N_2 \cdot {1 3 5} =N_2 \cdot {1 4 5}$.

Any other vertex $X$ of $P$ gives a smaller value for $N_2 \cdot X$.

\medskip
These two facets of $P$ generate orbits of length $60$ and $240$, respectively,
and these orbits cover all facets of $P$.

No vertex of $P$ is simultaneously joined by an edge of $P$
to a pair of antipodal vertices $(X \ X')$. This can be seen upon 
inspection of $\Delta_1$ and $\Delta_2$, and their antipodal facets. 
Therefore, by dividing out by the central involution $R$ we
obtain a triangulation of $\mathbb{R}P^4$ with 16 vertices
$1, 2, 3, 4, 5, 6$ and all ten pairs of disjoint triples among them.

Under the action of the remaining group $S_6$ there are
two orbits of 4-simplices of length 30 and 120, respectively.
This coincides with the observation made in \cite{Lu}, cf. the first part of Theorem~\ref{RP4}.
More precisely, one can establish the bijection between the two kinds
of vertices which is compatible with the two group actions $T$ and $S$ as done above.

Altogether, the proof of the last statement of Theorem~\ref{RP4} also leads
to an (independent) proof of the entire theorem.
\end{proof}

\section{Open problems}
\label{sec:problems}

We conclude this article by listing a number of open problems we came across while 
working on this topic. They are listed in no particular order.

\begin{enumerate}
  \item Is there a tight polyhedral embedding of $\mathbb{R}P^4$?

  \smallskip
  \noindent  The highly symmetric triangulation in Section~\ref{sec:min} does not seem
  to be suitable: One can put the $10$ distinguished vertices as an
  outer $9$-simplex into euclidean $9$-space, but the induced embedding
  is not tight. There cannot be any centrally symmetric triangulation because
  the Euler characteristic is odd. In particular, there is no analogue of
  Theorem~\ref{tightRP3}.
  
  \item Is there a tight polyhedral embedding of $\mathbb{C}P^3$ ?

  \smallskip
  \noindent  There does not seem to be a suitable candidate 
  as an abstract triangulation. An equilibrium triangulation
  containing the $15$-vertex $3$-torus is not suitable since this $15$-vertex
  $3$-torus is not a tight triangulation. 
  
  \item Is there a Hopf triangulation of $S^7$ generating a perfect
  equilibrium triangulation of~$\mathbb{C}P^4$?
  
  \smallskip
  \noindent
  Such a triangulation must have $31$ vertices to contain the central $4$-torus
  ${\bf 1 \ 2 \ 4 \ 8 \ 16 }$, as well as $5$-dimensional ``solid torus''
  ${\bf 1 \ 1 \ 1 \ 4 \ 8 \ 16 }$ of topological type $B^2 \times (S^1)^3$ and three of its multiples.
  Candidates for the $6$-dimensional ``solid tori'' of type $B^4 \times S^1 \times S^1$ exist as 
  permcycles ${\bf 1\ 1\ 1\ 4\ 4\ 4\ 16}$, and its orbit under $\sigma : x\mapsto 2\cdot x \mod 31$;
  and another ``solid torus'' triangulation, with $\mathbb{Z}_{31}$-generators listed in
  Appendix~\ref{app:d6}, and its orbit under $\sigma$.

\item Is there any (vertex-minimal) 22-vertex triangulation of $\mathbb{R}P^5$?

\item Is there an equilibrium triangulation of the quaternionic projective plane
along the lines of a Hopf triangulation of $S^7$ with
3-dimensional Hopf fibres and a ``central torus'' of type $S^3 \times S^3$
as a quotient of $S^3 \times S^3 \times S^3$ by the diagonal action
of $Sp(1) \cong S^3$ ?

\noindent There are several vertex-minimal 13-vertex triangulations of
$S^3 \times S^3$ but none
with a vertex transitive group action \cite{Lu}. Hence such a complex will be
rather unsymmetrical if it exists.
On the other hand there are fairly symmetric 15-vertex triangulations of
$S^3 \times S^3$ and also $S^7$ in \cite{KoLu}.
\end{enumerate}

\bigskip

Wolfgang K\"uhnel,
Department of Mathematics, 
University of Stuttgart, 
D--70550 Stuttgart

E-mail: {\tt kuehnel@mathematik.uni-stuttgart.de}

\medskip
Jonathan Spreer,
School of Mathematics and Statistics,
University of Sydney, Australia

E-mail: {\tt jonathan.spreer@sydney.edu.au}

\newpage

\appendix

\section{Decomposition of $\operatorname{4C}(1,2,4,8;31)$ into handlebodies $A_i$, $1\leq i \leq 4$}
\label{app:s7}

The generators of $\mathbb{Z}_{31}$-orbits of $A_1$ under cyclic shift $i \mapsto (i+1) \mod 31$:

\begin{tabular}{llll}
\\
$f$-vector: & \multicolumn{3}{l}{$(31, 465, 3255, 12710, 26722, 30132, 17236, 3937)$}\\
&&&\\
\hline
&&&\\
$(0\ 1\ 2\ 3\ 4\ 5\ 6\ 7 ),$&$( 0\ 1\ 2\ 3\ 4\ 5\ 7\ 8 ),$&$( 0\ 1\ 2\ 3\ 4\ 5\ 8\ 9 ),$&$( 0\ 1\ 2\ 3\ 4\ 5\ 9\ 10 )$\\ 
$(0\ 1\ 2\ 3\ 4\ 5\ 10\ 11 ),$&$( 0\ 1\ 2\ 3\ 4\ 5\ 11\ 12 ),$&$( 0\ 1\ 2\ 3\ 4\ 5\ 12\ 13 ),$&$( 0\ 1\ 2\ 3\ 4\ 5\ 13\ 14 )$\\ 
$(0\ 1\ 2\ 3\ 4\ 5\ 22\ 23 ),$&$( 0\ 1\ 2\ 3\ 4\ 5\ 23\ 24 ),$&$( 0\ 1\ 2\ 3\ 4\ 5\ 24\ 25 ),$&$( 0\ 1\ 2\ 3\ 4\ 5\ 25\ 26 )$\\ 
$(0\ 1\ 2\ 3\ 4\ 5\ 26\ 27 ),$&$( 0\ 1\ 2\ 3\ 4\ 5\ 27\ 28 ),$&$( 0\ 1\ 2\ 3\ 4\ 5\ 28\ 29 ),$&$( 0\ 1\ 2\ 3\ 5\ 6\ 7\ 8 )$\\ 
$(0\ 1\ 2\ 3\ 5\ 6\ 8\ 9 ),$&$( 0\ 1\ 2\ 3\ 5\ 6\ 9\ 10 ),$&$( 0\ 1\ 2\ 3\ 5\ 6\ 10\ 11 ),$&$( 0\ 1\ 2\ 3\ 5\ 6\ 11\ 12 )$\\ 
$(0\ 1\ 2\ 3\ 5\ 6\ 12\ 13 ),$&$( 0\ 1\ 2\ 3\ 5\ 6\ 13\ 14 ),$&$( 0\ 1\ 2\ 3\ 5\ 6\ 22\ 23 ),$&$( 0\ 1\ 2\ 3\ 5\ 6\ 23\ 24 )$\\ 
$(0\ 1\ 2\ 3\ 5\ 6\ 24\ 25 ),$&$( 0\ 1\ 2\ 3\ 5\ 6\ 25\ 26 ),$&$( 0\ 1\ 2\ 3\ 5\ 6\ 26\ 27 ),$&$( 0\ 1\ 2\ 3\ 5\ 6\ 27\ 28 )$\\ 
$(0\ 1\ 2\ 3\ 5\ 6\ 28\ 29 ),$&$( 0\ 1\ 2\ 3\ 6\ 7\ 8\ 9 ),$&$( 0\ 1\ 2\ 3\ 6\ 7\ 9\ 10 ),$&$( 0\ 1\ 2\ 3\ 6\ 7\ 10\ 11 )$\\ 
$(0\ 1\ 2\ 3\ 6\ 7\ 11\ 12 ),$&$( 0\ 1\ 2\ 3\ 6\ 7\ 12\ 13 ),$&$( 0\ 1\ 2\ 3\ 6\ 7\ 13\ 14 ),$&$( 0\ 1\ 2\ 3\ 6\ 7\ 23\ 24 )$\\ 
$(0\ 1\ 2\ 3\ 6\ 7\ 24\ 25 ),$&$( 0\ 1\ 2\ 3\ 6\ 7\ 25\ 26 ),$&$( 0\ 1\ 2\ 3\ 6\ 7\ 26\ 27 ),$&$( 0\ 1\ 2\ 3\ 6\ 7\ 27\ 28 )$\\ 
$(0\ 1\ 2\ 3\ 6\ 7\ 28\ 29 ),$&$( 0\ 1\ 2\ 3\ 7\ 8\ 9\ 10 ),$&$( 0\ 1\ 2\ 3\ 7\ 8\ 10\ 11 ),$&$( 0\ 1\ 2\ 3\ 7\ 8\ 11\ 15 )$\\ 
$(0\ 1\ 2\ 3\ 7\ 8\ 26\ 27 ),$&$( 0\ 1\ 2\ 3\ 7\ 8\ 27\ 28 ),$&$( 0\ 1\ 2\ 3\ 7\ 8\ 28\ 29 ),$&$( 0\ 1\ 2\ 3\ 7\ 11\ 12\ 14 )$\\ 
$(0\ 1\ 2\ 3\ 7\ 11\ 14\ 15 ),$&$( 0\ 1\ 2\ 3\ 8\ 9\ 27\ 28 ),$&$( 0\ 1\ 2\ 3\ 8\ 9\ 28\ 29 ),$&$( 0\ 1\ 2\ 3\ 9\ 10\ 27\ 28 )$\\ 
$(0\ 1\ 2\ 3\ 9\ 10\ 28\ 29 ),$&$( 0\ 1\ 2\ 3\ 10\ 11\ 27\ 28 ),$&$( 0\ 1\ 2\ 3\ 10\ 11\ 28\ 29 ),$&$( 0\ 1\ 2\ 3\ 11\ 12\ 28\ 29 )$\\ 
$(0\ 1\ 2\ 3\ 19\ 20\ 23\ 27 ),$&$( 0\ 1\ 2\ 3\ 19\ 23\ 26\ 27 ),$&$( 0\ 1\ 2\ 3\ 20\ 21\ 27\ 28 ),$&$( 0\ 1\ 2\ 3\ 20\ 21\ 28\ 29 )$\\ 
$(0\ 1\ 2\ 3\ 20\ 22\ 23\ 27 ),$&$( 0\ 1\ 2\ 3\ 21\ 22\ 27\ 28 ),$&$( 0\ 1\ 2\ 3\ 21\ 22\ 28\ 29 ),$&$( 0\ 1\ 2\ 3\ 22\ 23\ 27\ 28 )$\\ 
$(0\ 1\ 2\ 3\ 22\ 23\ 28\ 29 ),$&$( 0\ 1\ 2\ 3\ 23\ 24\ 26\ 27 ),$&$( 0\ 1\ 2\ 3\ 23\ 24\ 27\ 28 ),$&$( 0\ 1\ 2\ 3\ 23\ 24\ 28\ 29 )$\\ 
$(0\ 1\ 2\ 3\ 24\ 25\ 27\ 28 ),$&$( 0\ 1\ 2\ 3\ 24\ 25\ 28\ 29 ),$&$( 0\ 1\ 2\ 3\ 25\ 26\ 28\ 29 ),$&$( 0\ 1\ 2\ 6\ 7\ 10\ 11\ 13 )$\\ 
$(0\ 1\ 2\ 6\ 7\ 10\ 13\ 14 ),$&$( 0\ 1\ 2\ 6\ 7\ 11\ 12\ 13 ),$&$( 0\ 1\ 2\ 7\ 8\ 10\ 11\ 14 ),$&$( 0\ 1\ 2\ 7\ 8\ 11\ 14\ 15 )$\\ 
$(0\ 1\ 2\ 7\ 10\ 11\ 13\ 14 ),$&$( 0\ 1\ 2\ 8\ 10\ 11\ 14\ 15 ),$&$( 0\ 1\ 2\ 18\ 19\ 22\ 23\ 25 ),$&$( 0\ 1\ 2\ 18\ 19\ 22\ 25\ 26 )$\\ 
$(0\ 1\ 2\ 19\ 20\ 22\ 23\ 26 ),$&$( 0\ 1\ 2\ 19\ 20\ 23\ 26\ 27 ),$&$( 0\ 1\ 2\ 19\ 22\ 23\ 25\ 26 ),$&$( 0\ 1\ 2\ 20\ 22\ 23\ 26\ 27 )$\\ 
$(0\ 1\ 3\ 4\ 6\ 7\ 9\ 10 ),$&$( 0\ 1\ 3\ 4\ 6\ 7\ 10\ 11 ),$&$( 0\ 1\ 3\ 4\ 6\ 7\ 11\ 12 ),$&$( 0\ 1\ 3\ 4\ 6\ 7\ 12\ 13 )$\\ 
$(0\ 1\ 3\ 4\ 6\ 7\ 13\ 14 ),$&$( 0\ 1\ 3\ 4\ 6\ 7\ 14\ 15 ),$&$( 0\ 1\ 3\ 4\ 6\ 7\ 23\ 24 ),$&$( 0\ 1\ 3\ 4\ 6\ 7\ 24\ 25 )$\\ 
$(0\ 1\ 3\ 4\ 6\ 7\ 25\ 26 ),$&$( 0\ 1\ 3\ 4\ 6\ 7\ 26\ 27 ),$&$( 0\ 1\ 3\ 4\ 6\ 7\ 27\ 28 ),$&$( 0\ 1\ 3\ 4\ 7\ 8\ 10\ 11 )$\\ 
$(0\ 1\ 3\ 4\ 7\ 8\ 11\ 12 ),$&$( 0\ 1\ 3\ 4\ 7\ 8\ 24\ 26 ),$&$( 0\ 1\ 3\ 4\ 7\ 8\ 26\ 27 ),$&$( 0\ 1\ 3\ 4\ 7\ 8\ 27\ 28 )$\\ 
$(0\ 1\ 3\ 4\ 8\ 9\ 11\ 27 ),$&$( 0\ 1\ 3\ 4\ 8\ 9\ 27\ 28 ),$&$( 0\ 1\ 3\ 4\ 8\ 24\ 26\ 27 ),$&$( 0\ 1\ 3\ 4\ 9\ 11\ 27\ 28 )$\\ 
$(0\ 1\ 3\ 4\ 23\ 24\ 27\ 28 ),$&$( 0\ 1\ 3\ 7\ 8\ 11\ 12\ 14 ),$&$( 0\ 1\ 3\ 7\ 8\ 11\ 14\ 15 ),$&$( 0\ 1\ 3\ 7\ 8\ 12\ 14\ 15 )$\\ 
$(0\ 1\ 4\ 5\ 9\ 11\ 12\ 29 ),$&$( 0\ 1\ 4\ 7\ 8\ 12\ 14\ 15 ),$&$( 0\ 1\ 2\ 3\ 5\ 6\ 14\ 15 ),$&$( 0\ 1\ 2\ 3\ 6\ 7\ 14\ 15 )$\\ 
$(0\ 1\ 2\ 3\ 7\ 12\ 13\ 14 ),$&$( 0\ 1\ 2\ 3\ 8\ 10\ 11\ 15 ),$&$( 0\ 1\ 3\ 4\ 7\ 8\ 12\ 15 ),$&$( 0\ 1\ 3\ 4\ 7\ 12\ 14\ 15 )$\\ 
$(0\ 1\ 2\ 3\ 19\ 20\ 28\ 29 ),$&$( 0\ 1\ 2\ 3\ 19\ 20\ 27\ 28 ),$&$( 0\ 1\ 2\ 3\ 20\ 21\ 22\ 27 ),$&$( 0\ 1\ 2\ 3\ 19\ 23\ 24\ 26 )$\\ 
$(0\ 1\ 3\ 4\ 20\ 23\ 27\ 28 ),$&$( 0\ 1\ 3\ 4\ 20\ 21\ 23\ 28 ),$&$( 0\ 1\ 3\ 4\ 8\ 11\ 12\ 15 ),$&$( 0\ 1\ 3\ 4\ 20\ 23\ 24\ 27 )$\\ 
$(0\ 1\ 3\ 19\ 20\ 23\ 24\ 27 ),$&$( 0\ 1\ 4\ 5\ 8\ 9\ 12\ 28 ),$&$( 0\ 1\ 4\ 5\ 21\ 23\ 24\ 28 )$&
\end{tabular}

\newpage
 
Orbit representatives of $A_2$:

\begin{tabular}{llll}
\\
$f$-vector: & \multicolumn{3}{l}{$(31, 465, 3255, 11842, 23281, 25017, 13857, 3100)$}\\
&&&\\
\hline
&&&\\
$(0\ 1\ 2\ 3\ 4\ 5\ 14\ 15 ),$&$( 0\ 1\ 2\ 3\ 4\ 5\ 15\ 16 ),$&$( 0\ 1\ 2\ 3\ 4\ 5\ 16\ 20 ),$&$( 0\ 1\ 2\ 3\ 4\ 5\ 20\ 21 )$\\ 
$(0\ 1\ 2\ 3\ 4\ 5\ 21\ 22 ),$&$( 0\ 1\ 2\ 3\ 4\ 15\ 16\ 17 ),$&$( 0\ 1\ 2\ 3\ 4\ 15\ 17\ 19 ),$&$( 0\ 1\ 2\ 3\ 4\ 16\ 17\ 19 )$\\ 
$(0\ 1\ 2\ 3\ 4\ 16\ 18\ 19 ),$&$( 0\ 1\ 2\ 3\ 4\ 16\ 18\ 20 ),$&$( 0\ 1\ 2\ 3\ 4\ 18\ 19\ 20 ),$&$( 0\ 1\ 2\ 3\ 5\ 6\ 15\ 16 )$\\ 
$(0\ 1\ 2\ 3\ 5\ 6\ 16\ 21 ),$&$( 0\ 1\ 2\ 3\ 5\ 6\ 21\ 22 ),$&$( 0\ 1\ 2\ 3\ 5\ 16\ 18\ 20 ),$&$( 0\ 1\ 2\ 3\ 5\ 16\ 18\ 21 )$\\ 
$(0\ 1\ 2\ 3\ 5\ 18\ 20\ 21 ),$&$( 0\ 1\ 2\ 3\ 6\ 16\ 18\ 21 ),$&$( 0\ 1\ 2\ 3\ 6\ 16\ 18\ 22 ),$&$( 0\ 1\ 2\ 3\ 6\ 18\ 21\ 22 )$\\ 
$(0\ 1\ 2\ 3\ 7\ 16\ 18\ 22 ),$&$( 0\ 1\ 2\ 3\ 11\ 12\ 14\ 15 ),$&$( 0\ 1\ 2\ 3\ 11\ 12\ 15\ 16 ),$&$( 0\ 1\ 2\ 3\ 12\ 13\ 14\ 15 )$\\ 
$(0\ 1\ 2\ 3\ 12\ 13\ 15\ 16 ),$&$( 0\ 1\ 2\ 3\ 12\ 13\ 16\ 28 ),$&$( 0\ 1\ 2\ 3\ 12\ 13\ 28\ 29 ),$&$( 0\ 1\ 2\ 3\ 12\ 16\ 18\ 27 )$\\ 
$(0\ 1\ 2\ 3\ 12\ 16\ 18\ 28 ),$&$( 0\ 1\ 2\ 3\ 13\ 14\ 15\ 16 ),$&$( 0\ 1\ 2\ 3\ 13\ 14\ 16\ 29 ),$&$( 0\ 1\ 2\ 3\ 13\ 16\ 18\ 28 )$\\ 
$(0\ 1\ 2\ 3\ 13\ 16\ 18\ 29 ),$&$( 0\ 1\ 2\ 3\ 13\ 18\ 28\ 29 ),$&$( 0\ 1\ 2\ 3\ 14\ 16\ 18\ 29 ),$&$( 0\ 1\ 2\ 3\ 15\ 16\ 17\ 18 )$\\ 
$(0\ 1\ 2\ 3\ 18\ 19\ 21\ 22 ),$&$( 0\ 1\ 2\ 3\ 18\ 19\ 22\ 23 ),$&$( 0\ 1\ 2\ 3\ 18\ 19\ 28\ 29 ),$&$( 0\ 1\ 2\ 3\ 19\ 20\ 22\ 23 )$\\ 
$(0\ 1\ 2\ 4\ 5\ 15\ 16\ 20 ),$&$( 0\ 1\ 2\ 4\ 15\ 16\ 17\ 20 ),$&$( 0\ 1\ 2\ 4\ 16\ 17\ 19\ 20 ),$&$( 0\ 1\ 2\ 4\ 16\ 18\ 19\ 20 )$\\ 
$(0\ 1\ 2\ 5\ 6\ 15\ 16\ 21 ),$&$( 0\ 1\ 2\ 5\ 15\ 16\ 17\ 20 ),$&$( 0\ 1\ 2\ 5\ 15\ 16\ 17\ 21 ),$&$( 0\ 1\ 2\ 5\ 17\ 18\ 20\ 21 )$\\ 
$(0\ 1\ 2\ 6\ 15\ 16\ 17\ 21 ),$&$( 0\ 1\ 2\ 6\ 17\ 18\ 21\ 22 ),$&$( 0\ 1\ 2\ 11\ 12\ 15\ 16\ 27 ),$&$( 0\ 1\ 2\ 11\ 15\ 16\ 17\ 27 )$\\ 
$(0\ 1\ 2\ 12\ 13\ 15\ 16\ 28 ),$&$( 0\ 1\ 2\ 12\ 15\ 16\ 17\ 27 ),$&$( 0\ 1\ 2\ 12\ 15\ 16\ 17\ 28 ),$&$( 0\ 1\ 2\ 12\ 17\ 18\ 27\ 28 )$\\ 
$(0\ 1\ 2\ 13\ 14\ 16\ 17\ 29 ),$&$( 0\ 1\ 2\ 13\ 15\ 16\ 17\ 28 ),$&$( 0\ 1\ 2\ 13\ 17\ 18\ 28\ 29 ),$&$( 0\ 1\ 3\ 4\ 6\ 16\ 19\ 21 )$\\ 
$(0\ 1\ 3\ 4\ 6\ 16\ 19\ 22 ),$&$( 0\ 1\ 3\ 4\ 6\ 19\ 21\ 22 ),$&$( 0\ 1\ 3\ 4\ 7\ 16\ 19\ 22 ),$&$( 0\ 1\ 3\ 4\ 12\ 13\ 15\ 16 )$\\ 
$(0\ 1\ 3\ 4\ 13\ 14\ 16\ 29 ),$&$( 0\ 1\ 3\ 4\ 13\ 16\ 19\ 28 ),$&$( 0\ 1\ 3\ 4\ 13\ 16\ 19\ 29 ),$&$( 0\ 1\ 3\ 4\ 14\ 16\ 19\ 29 )$\\ 
$(0\ 1\ 3\ 5\ 16\ 18\ 20\ 21 ),$&$( 0\ 1\ 3\ 6\ 16\ 18\ 21\ 22 ),$&$( 0\ 1\ 3\ 6\ 16\ 19\ 21\ 22 ),$&$( 0\ 1\ 3\ 12\ 16\ 18\ 27\ 28 )$\\ 
$(0\ 1\ 3\ 16\ 18\ 19\ 22\ 23 ),$&$( 0\ 1\ 3\ 16\ 18\ 19\ 27\ 28 ),$&$( 0\ 1\ 4\ 5\ 14\ 16\ 20\ 29 ),$&$( 0\ 1\ 4\ 6\ 16\ 19\ 21\ 22 )$\\ 
$(0\ 1\ 2\ 3\ 6\ 7\ 15\ 16 ),$&$( 0\ 1\ 2\ 3\ 6\ 7\ 16\ 22 ),$&$( 0\ 1\ 2\ 6\ 7\ 16\ 17\ 22 ),$&$( 0\ 1\ 3\ 4\ 6\ 7\ 16\ 22 )$\\ 
$(0\ 1\ 2\ 3\ 18\ 19\ 27\ 28 ),$&$( 0\ 1\ 2\ 3\ 12\ 18\ 27\ 28 ),$&$( 0\ 1\ 2\ 11\ 16\ 17\ 26\ 27 ),$&$( 0\ 1\ 2\ 3\ 6\ 7\ 22\ 23 )$\\ 
$(0\ 1\ 2\ 3\ 7\ 18\ 22\ 23 ),$&$( 0\ 1\ 2\ 3\ 11\ 12\ 16\ 27 ),$&$( 0\ 1\ 2\ 3\ 11\ 12\ 27\ 28 ),$&$( 0\ 1\ 3\ 4\ 6\ 7\ 15\ 16 )$\\ 
$(0\ 1\ 3\ 4\ 6\ 7\ 22\ 23 ),$&$( 0\ 1\ 3\ 4\ 7\ 16\ 19\ 23 ),$&$( 0\ 1\ 3\ 4\ 7\ 19\ 22\ 23 ),$&$( 0\ 1\ 3\ 4\ 12\ 13\ 16\ 28 )$\\ 
$(0\ 1\ 3\ 4\ 12\ 16\ 19\ 28 ),$&$( 0\ 1\ 3\ 7\ 16\ 18\ 22\ 23 ),$&$( 0\ 1\ 3\ 7\ 16\ 19\ 22\ 23 ),$&$( 0\ 1\ 3\ 12\ 16\ 19\ 27\ 28 )$\\ 
$(0\ 1\ 4\ 5\ 13\ 16\ 20\ 29 ),$&$( 0\ 1\ 4\ 7\ 16\ 19\ 22\ 23 ),$&$( 0\ 1\ 4\ 7\ 16\ 20\ 22\ 23 ),$&$( 0\ 1\ 5\ 7\ 16\ 20\ 22\ 23 )$
\end{tabular}

\newpage

Orbit representatives of $A_3$:

\begin{tabular}{llll}
\\
$f$-vector: & \multicolumn{3}{l}{$(31, 465, 3131, 11067, 21235, 22289, 12059, 2635)$}\\
&&&\\
\hline
&&&\\
$(0\ 1\ 2\ 3\ 7\ 8\ 15\ 16 ),$&$( 0\ 1\ 2\ 3\ 7\ 8\ 16\ 23 ),$&$( 0\ 1\ 2\ 3\ 7\ 8\ 23\ 24 ),$&$( 0\ 1\ 2\ 3\ 8\ 10\ 15\ 16 )$\\ 
$(0\ 1\ 2\ 3\ 8\ 10\ 16\ 26 ),$&$( 0\ 1\ 2\ 3\ 8\ 16\ 18\ 23 ),$&$( 0\ 1\ 2\ 3\ 8\ 16\ 18\ 26 ),$&$( 0\ 1\ 2\ 3\ 8\ 18\ 23\ 24 )$\\ 
$(0\ 1\ 2\ 3\ 8\ 18\ 24\ 26 ),$&$( 0\ 1\ 2\ 3\ 10\ 11\ 16\ 26 ),$&$( 0\ 1\ 2\ 3\ 10\ 11\ 26\ 27 ),$&$( 0\ 1\ 2\ 3\ 11\ 16\ 18\ 26 )$\\ 
$(0\ 1\ 2\ 3\ 11\ 18\ 26\ 27 ),$&$( 0\ 1\ 2\ 3\ 18\ 19\ 24\ 26 ),$&$( 0\ 1\ 2\ 3\ 18\ 19\ 26\ 27 ),$&$( 0\ 1\ 2\ 6\ 7\ 15\ 17\ 22 )$\\ 
$(0\ 1\ 2\ 7\ 8\ 9\ 15\ 17 ),$&$( 0\ 1\ 2\ 7\ 8\ 9\ 17\ 23 ),$&$( 0\ 1\ 2\ 7\ 8\ 9\ 23\ 25 ),$&$( 0\ 1\ 2\ 7\ 8\ 10\ 14\ 15 )$\\ 
$(0\ 1\ 2\ 7\ 8\ 15\ 16\ 17 ),$&$( 0\ 1\ 2\ 7\ 8\ 16\ 17\ 23 ),$&$( 0\ 1\ 2\ 7\ 8\ 23\ 24\ 25 ),$&$( 0\ 1\ 2\ 7\ 9\ 15\ 17\ 25 )$\\ 
$(0\ 1\ 2\ 7\ 9\ 17\ 23\ 25 ),$&$( 0\ 1\ 2\ 7\ 16\ 17\ 22\ 23 ),$&$( 0\ 1\ 2\ 7\ 16\ 18\ 22\ 23 ),$&$( 0\ 1\ 2\ 8\ 9\ 10\ 16\ 25 )$\\ 
$(0\ 1\ 2\ 8\ 9\ 10\ 25\ 26 ),$&$( 0\ 1\ 2\ 8\ 9\ 15\ 16\ 17 ),$&$( 0\ 1\ 2\ 8\ 9\ 16\ 17\ 25 ),$&$( 0\ 1\ 2\ 8\ 9\ 17\ 23\ 25 )$\\ 
$(0\ 1\ 2\ 8\ 10\ 16\ 24\ 25 ),$&$( 0\ 1\ 2\ 8\ 10\ 16\ 24\ 26 ),$&$( 0\ 1\ 2\ 8\ 16\ 17\ 23\ 24 ),$&$( 0\ 1\ 2\ 8\ 16\ 17\ 24\ 25 )$\\ 
$(0\ 1\ 2\ 8\ 16\ 18\ 23\ 24 ),$&$( 0\ 1\ 2\ 8\ 16\ 18\ 24\ 26 ),$&$( 0\ 1\ 2\ 9\ 10\ 15\ 16\ 17 ),$&$( 0\ 1\ 2\ 9\ 10\ 15\ 17\ 25 )$\\ 
$(0\ 1\ 2\ 9\ 10\ 16\ 17\ 25 ),$&$( 0\ 1\ 2\ 10\ 11\ 15\ 17\ 26 ),$&$( 0\ 1\ 2\ 10\ 11\ 16\ 17\ 26 ),$&$( 0\ 1\ 2\ 10\ 16\ 17\ 25\ 26 )$\\ 
$(0\ 1\ 2\ 11\ 16\ 18\ 26\ 27 ),$&$( 0\ 1\ 2\ 18\ 19\ 23\ 25\ 26 ),$&$( 0\ 1\ 3\ 8\ 10\ 11\ 15\ 16 ),$&$( 0\ 1\ 3\ 8\ 10\ 11\ 16\ 26 )$\\ 
$(0\ 1\ 3\ 8\ 10\ 11\ 26\ 27 ),$&$( 0\ 1\ 3\ 8\ 11\ 16\ 24\ 26 ),$&$( 0\ 1\ 3\ 8\ 16\ 18\ 23\ 24 ),$&$( 0\ 1\ 3\ 8\ 16\ 18\ 24\ 26 )$\\ 
$(0\ 1\ 3\ 11\ 16\ 18\ 26\ 27 ),$&$( 0\ 1\ 3\ 16\ 18\ 19\ 24\ 26 ),$&$( 0\ 1\ 5\ 6\ 14\ 16\ 21\ 29 ),$&$( 0\ 1\ 6\ 8\ 14\ 16\ 24\ 29 )$\\ 
$(0\ 1\ 6\ 8\ 16\ 21\ 23\ 24 ),$&$( 0\ 1\ 6\ 8\ 16\ 21\ 24\ 29 ),$&$( 0\ 1\ 7\ 8\ 15\ 16\ 23\ 24 ),$&$( 0\ 1\ 7\ 9\ 15\ 17\ 23\ 25 )$\\ 
$(0\ 2\ 7\ 9\ 15\ 17\ 23\ 25 ),$&$( 0\ 1\ 2\ 3\ 7\ 8\ 24\ 25 ),$&$( 0\ 1\ 2\ 3\ 7\ 8\ 25\ 26 ),$&$( 0\ 1\ 2\ 3\ 7\ 16\ 18\ 23 )$\\ 
$(0\ 1\ 2\ 3\ 8\ 9\ 10\ 26 ),$&$( 0\ 1\ 2\ 3\ 10\ 11\ 15\ 16 ),$&$( 0\ 1\ 2\ 6\ 7\ 15\ 16\ 17 ),$&$( 0\ 1\ 2\ 7\ 9\ 10\ 14\ 15 )$\\ 
$(0\ 1\ 3\ 8\ 11\ 24\ 26\ 27 ),$&$( 0\ 1\ 3\ 8\ 16\ 19\ 23\ 24 ),$&$( 0\ 1\ 3\ 11\ 16\ 24\ 26\ 27 ),$&$( 0\ 1\ 3\ 16\ 18\ 19\ 23\ 24 )$\\ 
$(0\ 1\ 2\ 3\ 9\ 10\ 26\ 27 ),$&$( 0\ 1\ 2\ 3\ 8\ 9\ 26\ 27 ),$&$( 0\ 1\ 2\ 3\ 11\ 16\ 18\ 27 ),$&$( 0\ 1\ 2\ 3\ 8\ 24\ 25\ 26 )$\\ 
$(0\ 1\ 2\ 3\ 18\ 19\ 23\ 24 ),$&$( 0\ 1\ 2\ 10\ 11\ 15\ 16\ 17 ),$&$( 0\ 1\ 2\ 18\ 19\ 23\ 24\ 26 ),$&$( 0\ 1\ 3\ 16\ 19\ 24\ 26\ 27 )$\\ 
$(0\ 1\ 5\ 8\ 16\ 21\ 23\ 24 ),$&$( 0\ 1\ 3\ 16\ 18\ 19\ 26\ 27 ),$&$( 0\ 1\ 3\ 8\ 11\ 16\ 24\ 27 ),$&$( 0\ 1\ 5\ 8\ 16\ 21\ 24\ 29 )$\\ 
$(0\ 2\ 7\ 10\ 15\ 18\ 23\ 26 )$&&&
\end{tabular}

\bigskip

Orbit representatives of $A_4$:

\begin{tabular}{llll}
\\
$f$-vector: & \multicolumn{3}{l}{$(31, 465, 2728, 8153, 13547, 12679, 6262, 1271)$}\\
&&&\\
\hline
&&&\\
$(0\ 1\ 3\ 4\ 7\ 8\ 16\ 23 ),$&$( 0\ 1\ 3\ 4\ 8\ 11\ 15\ 16 ),$&$( 0\ 1\ 3\ 4\ 8\ 11\ 16\ 27 ),$&$( 0\ 1\ 3\ 4\ 8\ 16\ 19\ 23 )$\\ 
$(0\ 1\ 3\ 4\ 8\ 16\ 19\ 27 ),$&$( 0\ 1\ 3\ 4\ 8\ 19\ 23\ 24 ),$&$( 0\ 1\ 3\ 4\ 8\ 19\ 24\ 27 ),$&$( 0\ 1\ 3\ 4\ 11\ 12\ 16\ 27 )$\\ 
$(0\ 1\ 3\ 4\ 12\ 16\ 19\ 27 ),$&$( 0\ 1\ 3\ 4\ 12\ 19\ 27\ 28 ),$&$( 0\ 1\ 3\ 4\ 19\ 20\ 24\ 27 ),$&$( 0\ 1\ 3\ 8\ 16\ 19\ 24\ 27 )$\\ 
$(0\ 1\ 4\ 5\ 8\ 12\ 15\ 16 ),$&$( 0\ 1\ 4\ 5\ 8\ 12\ 16\ 28 ),$&$( 0\ 1\ 4\ 5\ 8\ 16\ 20\ 23 ),$&$( 0\ 1\ 4\ 5\ 8\ 16\ 20\ 28 )$\\ 
$(0\ 1\ 4\ 5\ 8\ 20\ 23\ 24 ),$&$( 0\ 1\ 4\ 5\ 8\ 20\ 24\ 28 ),$&$( 0\ 1\ 4\ 5\ 12\ 13\ 16\ 28 ),$&$( 0\ 1\ 4\ 5\ 13\ 16\ 20\ 28 )$\\ 
$(0\ 1\ 4\ 5\ 20\ 21\ 24\ 28 ),$&$( 0\ 1\ 4\ 8\ 11\ 12\ 16\ 27 ),$&$( 0\ 1\ 4\ 8\ 12\ 16\ 24\ 27 ),$&$( 0\ 1\ 4\ 8\ 12\ 16\ 24\ 28 )$\\ 
$(0\ 1\ 4\ 8\ 16\ 19\ 23\ 24 ),$&$( 0\ 1\ 4\ 8\ 16\ 19\ 24\ 27 ),$&$( 0\ 1\ 4\ 8\ 16\ 20\ 23\ 24 ),$&$( 0\ 1\ 4\ 8\ 16\ 20\ 24\ 28 )$\\ 
$(0\ 1\ 4\ 16\ 19\ 20\ 24\ 27 ),$&$( 0\ 1\ 5\ 8\ 13\ 16\ 24\ 28 ),$&$( 0\ 1\ 5\ 8\ 13\ 16\ 24\ 29 ),$&$( 0\ 1\ 5\ 8\ 16\ 20\ 23\ 24 )$\\ 
$(0\ 1\ 5\ 8\ 16\ 20\ 24\ 28 ),$&$( 0\ 3\ 7\ 10\ 15\ 18\ 23\ 26 ),$&$( 0\ 3\ 7\ 11\ 15\ 19\ 23\ 27 ),$&$( 0\ 1\ 3\ 4\ 7\ 8\ 15\ 16 )$\\ 
$(0\ 1\ 3\ 4\ 19\ 20\ 27\ 28 ),$&$( 0\ 1\ 3\ 4\ 7\ 8\ 23\ 24 ),$&$( 0\ 1\ 3\ 4\ 11\ 12\ 15\ 16 ),$&$( 0\ 1\ 3\ 4\ 11\ 12\ 27\ 28 )$\\ 
$(0\ 1\ 3\ 4\ 19\ 20\ 23\ 24 )$&&&
\end{tabular}

\section{Facet list of equilibrium triangulation of $\mathbb{R}P^4$}
\label{app:rp4}

\begin{tabular}{lllll}
\\
$f$-vector: & \multicolumn{4}{l}{$(21, 180, 520, 600, 240)$}\\
&&&&\\
\hline
&&&&\\
$( 1\ 4\ 6\ 7\ 17 )$,&$ ( 1\ 2\ 4\ 6\ 17 )$,&$ ( 1\ 5\ 6\ 7\ 17 )$,&$ ( 1\ 3\ 4\ 7\ 17 )$,&$ ( 4\ 6\ 7\ 8\ 17 )$,\\
$( 1\ 4\ 6\ 7\ 18 )$,&$ ( 1\ 2\ 4\ 6\ 18 )$,&$ ( 1\ 5\ 6\ 7\ 18 )$,&$ ( 1\ 3\ 4\ 7\ 18 )$,&$ ( 4\ 6\ 7\ 8\ 18 )$,\\
$( 9\ 12\ 14\ 15\ 17 )$,&$ ( 9\ 10\ 12\ 14\ 17 )$,&$ ( 9\ 13\ 14\ 15\ 17 )$,&$ ( 9\ 11\ 12\ 15\ 17 )$,&$ ( 12\ 14\ 15\ 16\ 17 )$,\\
$( 9\ 12\ 14\ 15\ 18 )$,&$ ( 9\ 10\ 12\ 14\ 18 )$,&$ ( 9\ 13\ 14\ 15\ 18 )$,&$ ( 9\ 11\ 12\ 15\ 18 )$,&$ ( 12\ 14\ 15\ 16\ 18 )$,\\
$( 1\ 4\ 10\ 11\ 17 )$,&$ ( 1\ 2\ 4\ 10\ 17 )$,&$ ( 1\ 9\ 10\ 11\ 17 )$,&$ ( 1\ 3\ 4\ 11\ 17 )$,&$ ( 4\ 10\ 11\ 12\ 17 )$,\\
$( 1\ 4\ 10\ 11\ 19 )$,&$ ( 1\ 2\ 4\ 10\ 19 )$,&$ ( 1\ 9\ 10\ 11\ 19 )$,&$ ( 1\ 3\ 4\ 11\ 19 )$,&$ ( 4\ 10\ 11\ 12\ 19 )$,\\
$( 5\ 8\ 14\ 15\ 17 )$,&$ ( 5\ 6\ 8\ 14\ 17 )$,&$ ( 5\ 13\ 14\ 15\ 17 )$,&$ ( 5\ 7\ 8\ 15\ 17 )$,&$ ( 8\ 14\ 15\ 16\ 17 )$,\\
$( 5\ 8\ 14\ 15\ 19 )$,&$ ( 5\ 6\ 8\ 14\ 19 )$,&$ ( 5\ 13\ 14\ 15\ 19 )$,&$ ( 5\ 7\ 8\ 15\ 19 )$,&$ ( 8\ 14\ 15\ 16\ 19 )$,\\
$( 1\ 6\ 10\ 13\ 17 )$,&$ ( 1\ 2\ 6\ 10\ 17 )$,&$ ( 1\ 9\ 10\ 13\ 17 )$,&$ ( 1\ 5\ 6\ 13\ 17 )$,&$ ( 6\ 10\ 13\ 14\ 17 )$,\\
$( 1\ 6\ 10\ 13\ 20 )$,&$ ( 1\ 2\ 6\ 10\ 20 )$,&$ ( 1\ 9\ 10\ 13\ 20 )$,&$ ( 1\ 5\ 6\ 13\ 20 )$,&$ ( 6\ 10\ 13\ 14\ 20 )$,\\
$( 3\ 8\ 12\ 15\ 17 )$,&$ ( 3\ 4\ 8\ 12\ 17 )$,&$ ( 3\ 11\ 12\ 15\ 17 )$,&$ ( 3\ 7\ 8\ 15\ 17 )$,&$ ( 8\ 12\ 15\ 16\ 17 )$,\\
$( 3\ 8\ 12\ 15\ 20 )$,&$ ( 3\ 4\ 8\ 12\ 20 )$,&$ ( 3\ 11\ 12\ 15\ 20 )$,&$ ( 3\ 7\ 8\ 15\ 20 )$,&$ ( 8\ 12\ 15\ 16\ 20 )$,\\
$( 1\ 7\ 11\ 13\ 17 )$,&$ ( 1\ 3\ 7\ 11\ 17 )$,&$ ( 1\ 9\ 11\ 13\ 17 )$,&$ ( 1\ 5\ 7\ 13\ 17 )$,&$ ( 7\ 11\ 13\ 15\ 17 )$,\\
$( 1\ 7\ 11\ 13\ 21 )$,&$ ( 1\ 3\ 7\ 11\ 21 )$,&$ ( 1\ 9\ 11\ 13\ 21 )$,&$ ( 1\ 5\ 7\ 13\ 21 )$,&$ ( 7\ 11\ 13\ 15\ 21 )$,\\
$( 2\ 8\ 12\ 14\ 17 )$,&$ ( 2\ 4\ 8\ 12\ 17 )$,&$ ( 2\ 10\ 12\ 14\ 17 )$,&$ ( 2\ 6\ 8\ 14\ 17 )$,&$ ( 8\ 12\ 14\ 16\ 17 )$,\\
$( 2\ 8\ 12\ 14\ 21 )$,&$ ( 2\ 4\ 8\ 12\ 21 )$,&$ ( 2\ 10\ 12\ 14\ 21 )$,&$ ( 2\ 6\ 8\ 14\ 21 )$,&$ ( 8\ 12\ 14\ 16\ 21 )$,\\
$( 1\ 4\ 14\ 15\ 18 )$,&$ ( 1\ 3\ 4\ 14\ 18 )$,&$ ( 1\ 2\ 4\ 15\ 18 )$,&$ ( 1\ 14\ 15\ 16\ 18 )$,&$ ( 4\ 13\ 14\ 15\ 18 )$,\\
$( 1\ 4\ 14\ 15\ 19 )$,&$ ( 1\ 3\ 4\ 14\ 19 )$,&$ ( 1\ 2\ 4\ 15\ 19 )$,&$ ( 1\ 14\ 15\ 16\ 19 )$,&$ ( 4\ 13\ 14\ 15\ 19 )$,\\
$( 5\ 8\ 10\ 11\ 18 )$,&$ ( 5\ 7\ 8\ 10\ 18 )$,&$ ( 5\ 6\ 8\ 11\ 18 )$,&$ ( 5\ 10\ 11\ 12\ 18 )$,&$ ( 8\ 9\ 10\ 11\ 18 )$,\\
$( 5\ 8\ 10\ 11\ 19 )$,&$ ( 5\ 7\ 8\ 10\ 19 )$,&$ ( 5\ 6\ 8\ 11\ 19 )$,&$ ( 5\ 10\ 11\ 12\ 19 )$,&$ ( 8\ 9\ 10\ 11\ 19 )$,\\
$( 1\ 6\ 12\ 15\ 18 )$,&$ ( 1\ 2\ 6\ 15\ 18 )$,&$ ( 1\ 12\ 15\ 16\ 18 )$,&$ ( 1\ 5\ 6\ 12\ 18 )$,&$ ( 6\ 11\ 12\ 15\ 18 )$,\\
$( 1\ 6\ 12\ 15\ 20 )$,&$ ( 1\ 2\ 6\ 15\ 20 )$,&$ ( 1\ 12\ 15\ 16\ 20 )$,&$ ( 1\ 5\ 6\ 12\ 20 )$,&$ ( 6\ 11\ 12\ 15\ 20 )$,\\
$( 3\ 8\ 10\ 13\ 18 )$,&$ ( 3\ 4\ 8\ 13\ 18 )$,&$ ( 3\ 10\ 13\ 14\ 18 )$,&$ ( 3\ 7\ 8\ 10\ 18 )$,&$ ( 8\ 9\ 10\ 13\ 18 )$,\\
$( 3\ 8\ 10\ 13\ 20 )$,&$ ( 3\ 4\ 8\ 13\ 20 )$,&$ ( 3\ 10\ 13\ 14\ 20 )$,&$ ( 3\ 7\ 8\ 10\ 20 )$,&$ ( 8\ 9\ 10\ 13\ 20 )$,\\
$( 1\ 7\ 12\ 14\ 18 )$,&$ ( 1\ 3\ 7\ 14\ 18 )$,&$ ( 1\ 12\ 14\ 16\ 18 )$,&$ ( 1\ 5\ 7\ 12\ 18 )$,&$ ( 7\ 10\ 12\ 14\ 18 )$,\\
$( 1\ 7\ 12\ 14\ 21 )$,&$ ( 1\ 3\ 7\ 14\ 21 )$,&$ ( 1\ 12\ 14\ 16\ 21 )$,&$ ( 1\ 5\ 7\ 12\ 21 )$,&$ ( 7\ 10\ 12\ 14\ 21 )$,\\
$( 2\ 8\ 11\ 13\ 18 )$,&$ ( 2\ 4\ 8\ 13\ 18 )$,&$ ( 2\ 11\ 13\ 15\ 18 )$,&$ ( 2\ 6\ 8\ 11\ 18 )$,&$ ( 8\ 9\ 11\ 13\ 18 )$,\\
$( 2\ 8\ 11\ 13\ 21 )$,&$ ( 2\ 4\ 8\ 13\ 21 )$,&$ ( 2\ 11\ 13\ 15\ 21 )$,&$ ( 2\ 6\ 8\ 11\ 21 )$,&$ ( 8\ 9\ 11\ 13\ 21 )$,\\
$( 1\ 8\ 10\ 15\ 19 )$,&$ ( 1\ 2\ 10\ 15\ 19 )$,&$ ( 1\ 8\ 15\ 16\ 19 )$,&$ ( 1\ 8\ 9\ 10\ 19 )$,&$ ( 7\ 8\ 10\ 15\ 19 )$,\\
$( 1\ 8\ 10\ 15\ 20 )$,&$ ( 1\ 2\ 10\ 15\ 20 )$,&$ ( 1\ 8\ 15\ 16\ 20 )$,&$ ( 1\ 8\ 9\ 10\ 20 )$,&$ ( 7\ 8\ 10\ 15\ 20 )$,\\
$( 3\ 6\ 12\ 13\ 19 )$,&$ ( 3\ 4\ 12\ 13\ 19 )$,&$ ( 3\ 6\ 13\ 14\ 19 )$,&$ ( 3\ 6\ 11\ 12\ 19 )$,&$ ( 5\ 6\ 12\ 13\ 19 )$,\\
$( 3\ 6\ 12\ 13\ 20 )$,&$ ( 3\ 4\ 12\ 13\ 20 )$,&$ ( 3\ 6\ 13\ 14\ 20 )$,&$ ( 3\ 6\ 11\ 12\ 20 )$,&$ ( 5\ 6\ 12\ 13\ 20 )$,\\
$( 1\ 8\ 11\ 14\ 19 )$,&$ ( 1\ 3\ 11\ 14\ 19 )$,&$ ( 1\ 8\ 14\ 16\ 19 )$,&$ ( 1\ 8\ 9\ 11\ 19 )$,&$ ( 6\ 8\ 11\ 14\ 19 )$,\\
$( 1\ 8\ 11\ 14\ 21 )$,&$ ( 1\ 3\ 11\ 14\ 21 )$,&$ ( 1\ 8\ 14\ 16\ 21 )$,&$ ( 1\ 8\ 9\ 11\ 21 )$,&$ ( 6\ 8\ 11\ 14\ 21 )$,\\
$( 2\ 7\ 12\ 13\ 19 )$,&$ ( 2\ 4\ 12\ 13\ 19 )$,&$ ( 2\ 7\ 13\ 15\ 19 )$,&$ ( 2\ 7\ 10\ 12\ 19 )$,&$ ( 5\ 7\ 12\ 13\ 19 )$,\\
$( 2\ 7\ 12\ 13\ 21 )$,&$ ( 2\ 4\ 12\ 13\ 21 )$,&$ ( 2\ 7\ 13\ 15\ 21 )$,&$ ( 2\ 7\ 10\ 12\ 21 )$,&$ ( 5\ 7\ 12\ 13\ 21 )$,\\
$( 1\ 8\ 12\ 13\ 20 )$,&$ ( 1\ 5\ 12\ 13\ 20 )$,&$ ( 1\ 8\ 12\ 16\ 20 )$,&$ ( 1\ 8\ 9\ 13\ 20 )$,&$ ( 4\ 8\ 12\ 13\ 20 )$,\\
$( 1\ 8\ 12\ 13\ 21 )$,&$ ( 1\ 5\ 12\ 13\ 21 )$,&$ ( 1\ 8\ 12\ 16\ 21 )$,&$ ( 1\ 8\ 9\ 13\ 21 )$,&$ ( 4\ 8\ 12\ 13\ 21 )$,\\
$( 2\ 7\ 11\ 14\ 20 )$,&$ ( 2\ 6\ 11\ 14\ 20 )$,&$ ( 2\ 7\ 11\ 15\ 20 )$,&$ ( 2\ 7\ 10\ 14\ 20 )$,&$ ( 3\ 7\ 11\ 14\ 20 )$,\\
$( 2\ 7\ 11\ 14\ 21 )$,&$ ( 2\ 6\ 11\ 14\ 21 )$,&$ ( 2\ 7\ 11\ 15\ 21 )$,&$ ( 2\ 7\ 10\ 14\ 21 )$,&$ ( 3\ 7\ 11\ 14\ 21 )$,\\
$( 5\ 6\ 7\ 8\ 17 )$,&$ ( 5\ 6\ 7\ 8\ 18 )$,&$ ( 9\ 10\ 11\ 12\ 17 )$,&$ ( 9\ 10\ 11\ 12\ 18 )$,&$ ( 3\ 4\ 7\ 8\ 17 )$,\\
$( 3\ 4\ 7\ 8\ 18 )$,&$ ( 9\ 10\ 13\ 14\ 17 )$,&$ ( 9\ 10\ 13\ 14\ 18 )$,&$ ( 2\ 4\ 6\ 8\ 17 )$,&$ ( 2\ 4\ 6\ 8\ 18 )$,\\
$( 9\ 11\ 13\ 15\ 17 )$,&$ ( 9\ 11\ 13\ 15\ 18 )$,&$ ( 3\ 4\ 11\ 12\ 17 )$,&$ ( 3\ 4\ 11\ 12\ 19 )$,&$ ( 5\ 6\ 13\ 14\ 17 )$,\\
$( 5\ 6\ 13\ 14\ 19 )$,&$ ( 2\ 4\ 10\ 12\ 17 )$,&$ ( 2\ 4\ 10\ 12\ 19 )$,&$ ( 5\ 7\ 13\ 15\ 17 )$,&$ ( 5\ 7\ 13\ 15\ 19 )$,\\
$( 2\ 6\ 10\ 14\ 17 )$,&$ ( 2\ 6\ 10\ 14\ 20 )$,&$ ( 3\ 7\ 11\ 15\ 17 )$,&$ ( 3\ 7\ 11\ 15\ 20 )$,&$ ( 3\ 4\ 13\ 14\ 18 )$,\\
$( 3\ 4\ 13\ 14\ 19 )$,&$ ( 5\ 6\ 11\ 12\ 18 )$,&$ ( 5\ 6\ 11\ 12\ 19 )$,&$ ( 2\ 4\ 13\ 15\ 18 )$,&$ ( 2\ 4\ 13\ 15\ 19 )$,\\
$( 5\ 7\ 10\ 12\ 18 )$,&$ ( 5\ 7\ 10\ 12\ 19 )$,&$ ( 2\ 6\ 11\ 15\ 18 )$,&$ ( 2\ 6\ 11\ 15\ 20 )$,&$ ( 3\ 7\ 10\ 14\ 18 )$,\\
$( 3\ 7\ 10\ 14\ 20 )$,&$ ( 2\ 7\ 10\ 15\ 19 )$,&$ ( 2\ 7\ 10\ 15\ 20 )$,&$ ( 3\ 6\ 11\ 14\ 19 )$,&$ ( 3\ 6\ 11\ 14\ 20 )$
\end{tabular}

\section{Candidate for $6$-dimensional ``solid torus'' for Hopf triangulation of $7$-sphere}
\label{app:d6}
  
Below are the generators of the $\mathbb{Z}_{31}$-orbits of the second set of candidates for the $6$-dimensional 
subsets of a Hopf triangulation of $S^7$ generating a perfect equilibrium triangulation of $\mathbb{C}P^4$.

The other candidates can be obtained by iteratively applying 
\[\sigma=(1,2,4,8,16)(3,6,12,24,17)(5,10,20,9,18)(7,14,28,25,19)(11,22,13,26,21)(15,30,29,27,23)\]
to the complex.
  
  \begin{tabular}{lllll}
\\
$f$-vector: & \multicolumn{4}{l}{$( 31, 465, 2294, 5332, 6417, 3875, 930)$}\\
&&&&\\
\hline
&&&&\\
  
  $( 0\ 1\ 2\ 3\ 4\ 5\ 13 ),$&$( 0\ 1\ 2\ 3\ 4\ 5\ 23 ),$&$( 0\ 1\ 2\ 3\ 4\ 12\ 13 ),$&$( 0\ 1\ 2\ 3\ 4\ 22\ 23 ),$&$( 0\ 1\ 2\ 3\ 5\ 7\ 15 )$\\
$( 0\ 1\ 2\ 3\ 5\ 7\ 23 ),$&$( 0\ 1\ 2\ 3\ 5\ 13\ 15 ),$&$( 0\ 1\ 2\ 3\ 11\ 12\ 13 ),$&$( 0\ 1\ 2\ 3\ 11\ 13\ 15 ),$&$( 0\ 1\ 2\ 3\ 11\ 27\ 29 )$\\ 
$( 0\ 1\ 2\ 3\ 19\ 21\ 23 ),$&$( 0\ 1\ 2\ 3\ 19\ 21\ 29 ),$&$( 0\ 1\ 2\ 3\ 19\ 27\ 29 ),$&$( 0\ 1\ 2\ 3\ 21\ 22\ 23 ),$&$( 0\ 1\ 2\ 4\ 5\ 6\ 14 )$\\ 
$( 0\ 1\ 2\ 4\ 5\ 6\ 23 ),$&$( 0\ 1\ 2\ 4\ 5\ 13\ 14 ),$&$( 0\ 1\ 2\ 4\ 6\ 22\ 23 ),$&$( 0\ 1\ 2\ 4\ 12\ 13\ 14 ),$&$( 0\ 1\ 2\ 5\ 6\ 7\ 15 )$\\
$( 0\ 1\ 2\ 5\ 6\ 7\ 23 ),$&$( 0\ 1\ 2\ 5\ 6\ 14\ 15 ),$&$( 0\ 1\ 2\ 5\ 13\ 14\ 15 ),$&$( 0\ 1\ 2\ 10\ 11\ 13\ 15 ),$&$( 0\ 1\ 2\ 10\ 11\ 27\ 29 )$\\ 
$( 0\ 1\ 2\ 10\ 12\ 13\ 14 ),$&$( 0\ 1\ 2\ 10\ 13\ 14\ 15 ),$&$( 0\ 1\ 2\ 18\ 19\ 27\ 28 ),$&$( 0\ 1\ 2\ 18\ 20\ 22\ 23 ),$&$( 0\ 1\ 2\ 19\ 20\ 28\ 29 )$
\end{tabular}

\end{document}